\documentclass[12pt]{amsart}

\usepackage{graphicx,color,enumerate,amssymb}
\usepackage{amsfonts,amsthm,amsmath,amssymb,graphicx,tabularx,epsf}%booktabs,
\usepackage{listings}
\usepackage{url}
\usepackage{pgf,tikz,pgfplots}
\pgfplotsset{compat=1.14}
\usepackage{mathrsfs}
\usetikzlibrary{arrows}
\usepackage{graphics}
\usepackage[latin1]{inputenc}
\usepackage[T1]{fontenc}
\usepackage[english]{babel}
\usepackage{float}
\theoremstyle{plain}% Theorem-like structures provided by amsthm.sty
\newtheorem{theorem}{Theorem}[section]
\newtheorem{lemma}[theorem]{Lemma}

\theoremstyle{definition}

\newtheorem{example}[theorem]{Example}

\theoremstyle{remark}
\newtheorem{remark}{Remark}

\usepackage[body={21.6cm,27.9cm},top=3cm,left=3cm,right=3cm,bottom=3cm]{geometry}
\usepackage{amsfonts}
\usepackage{amssymb}
\usepackage{amsmath}
\usepackage{tikz}
   \usetikzlibrary{positioning,arrows,calc}
   \tikzset{
   modal/.style={>=stealth,shorten >=1pt,shorten <=1pt,auto,node distance=1.5cm,
   semithick},
   world/.style={circle,draw,minimum size=0.5cm,fill=gray!15},
   point/.style={circle,draw,inner sep=0.5mm,fill=black},
   reflexive above/.style={->,loop,looseness=7,in=120,out=60},
   reflexive below/.style={->,loop,looseness=7,in=240,out=300},
   reflexive left/.style={->,loop,looseness=7,in=150,out=210},
   reflexive right/.style={->,loop,looseness=7,in=30,out=330}
   }
\usetikzlibrary{decorations.markings}
\definecolor{cadmiumgreen}{rgb}{0.0, 0.42, 0.24}

\newcommand{\ceilof}{\lceil\frac{n}{2} \rceil}
\newcommand{\avgdeg}[1]{2 - \frac{2}{#1} }

\begin{document}
\pagestyle{myheadings}

\title[Spectral ordering 2-switch]{Spectral ordering and 2-switch transformations}
\author[E. Oliveira]{Elismar Oliveira}
\address{Instituto de Matem\'atica e Estat\'{\i}stica, UFRGS, Porto Alegre, Brazil}
\email{\tt elismar.oliveira@ufrgs.br}
\author[V. Schv\"ollner]{Victor N. Schv\"ollner}
\address{Instituto de Matem\'atica Aplicada San Luis (UNSL-CONICET), Universidad Nacional de San Luis}
\email{\tt victor.schvollner.tag@gmail.com}
\author[V. Trevisan]{Vilmar Trevisan}
\address{Instituto de Matem\'atica e Estat\'{\i}stica, UFRGS, Porto Alegre, Brazil \and Department of Mathematics and Applications, University of Naples Federico II, Italy}
\email{\tt trevisan@mat.ufrgs.br}

\begin{abstract}
We address the problem of ordering trees with the same degree  sequence by their spectral radii. To achieve that, we consider 2-switch transformations which preserve the degree sequence and establish when the index decreases. Our main contribution is to determine a total ordering of a particular family by their indices according to a given parameter related to sizes in the tree.
\end{abstract}

\maketitle

\section{Introduction}
\label{sec:intro}
Giving a graphical degree sequence, a general, natural, and well studied  problem is to determine, with respect to a given parameter, the extremal members in the family of graphs satisfying this degree sequence. A 2-switch transformation (see definition below) is a convenient way to study this problem since it is a well known fact that for two graphs with the same degree sequence, one can be obtained from the other by applying successive 2-switches.

The main purpose of this note is to address the problem of finding extremal members in families of graphs having the same degree sequence, with respect to the \emph{spectral radius (or index)}, which is the largest eigenvalue of the adjacency matrix.  We remark that this problem has been studied in great generality in the celebre paper by T. Biyiko\u{g}lu and J. Leydold \cite{Biy08}, where the authors show that, in the maximum element, the degree sequence is non-increasing with respect to an ordering of the vertices induced by breadth-first search that is consistent with the eigenvector associated with the index. In \cite{BeLiSi06}, the authors determined the tree having maximum spectral radius among all caterpillar with a fixed degree sequence. We observe that in both papers 2-switch transformations is used to analyse the variation of the index.

In order to explain our results, we need a few definitions.
For a graph $G=(V,E)$ having four distinct vertices $a,b,c,d\in V$ such that $ab,cd\in E$ and $ac,bd\notin E$, the removal of the edges $ab$ and $cd$ from $G$ and the addition of  $ac$ and $bd$ to $G$ is referred to as a \emph{ 2-switch} in $G$. This is a well studied classical operation (see, for example \cite{Berge, FHM}). It is  straightforward to check that 2-switch operations preserve the  degree sequence.

As a way to illustrate how this class of problems may be approached, we study how the spectral radius varies upon 2-switch transformations in the family $\mathfrak{F}(n)$ of trees given in Figure \ref{main_fam}. The technique we use is a powerful algorithmic tool that allows one to compare the indices of two trees without computing them. Our main result is a total ordering in this family and, as a consequence, we obtain the extremal members.

We believe that this result is remarkable, since it is quite unusual to obtain a total order by any graph parameter. Spectral parameters have being used to classify many families, however it is rare that a total order is obtained. As examples, we refer to the papers \cite{BeLiSi2,BeOlTr,Cha03,Hof97,Lin2006,Oli18,Zha02}, where the ordering of graphs by the index is studied.

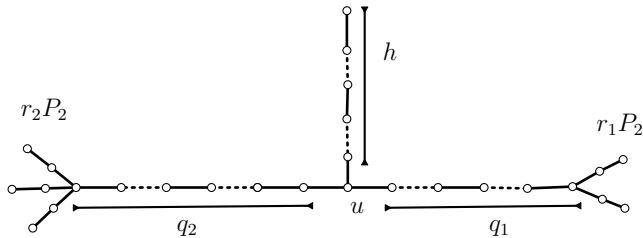
\begin{figure}[H]
  \centering
\definecolor{ffffff}{rgb}{1,1,1}
\definecolor{cqcqcq}{rgb}{0.7529411764705882,0.7529411764705882,0.7529411764705882}
\begin{tikzpicture}[line cap=round,line join=round,>=triangle 45,x=1cm,y=1cm,scale=0.6, every node/.style={scale=0.8}]
\clip(-9,-3.0) rectangle (9,3.3);
\draw [line width=1pt] (5.68,-1.24)-- (6.18,-1.44);
\draw [line width=1pt] (5.62,-0.62)-- (6.14,-0.36);
\draw [line width=1pt] (0.04,-0.98)-- (1.02,-0.98);
\draw [line width=1pt] (-5.98,-0.98)-- (-4.98,-0.98);
\draw [line width=1pt] (-6.48,-1.44)-- (-6.94,-1.88);
\draw [line width=1pt] (4.02,-1)-- (5.02,-0.96);
\draw [line width=1pt] (5.02,-0.96)-- (5.68,-1.24);
\draw [line width=1pt] (5.02,-0.96)-- (5.62,-0.62);
\draw [line width=1pt,dotted] (-3.98,-0.98)-- (-4.98,-0.98);
\draw (-0.1,-1.1) node[anchor=north west] {$u$};
\draw [line width=1pt] (-5.98,-0.98)-- (-6.48,-1.44);
\draw [line width=1pt] (0.02,2.94)-- (0.02,2.06);
\draw [line width=1pt] (0.04,1.3)-- (0.02,0.54);
\draw [line width=1pt,dotted] (-2.98,-0.98)-- (-1.96,-0.98);
\draw [line width=1pt] (-7.4,-1.02)-- (-6.66,-1);
\draw [line width=1pt] (-7.06,-0.12)-- (-6.52,-0.54);
\draw [line width=1pt] (-6.52,-0.54)-- (-5.98,-0.98);
\draw [line width=1pt] (-6.66,-1)-- (-5.98,-0.98);
\draw (5.36,0.84) node[anchor=north west] {$r_1 P_2$};
\draw (-7.4,1.18) node[anchor=north west] {$r_2 P_2$};
\draw [line width=1pt,dotted] (0.02,2.06)-- (0.04,1.3);
\draw [line width=1pt] (0.04,-0.3)-- (0.04,-0.98);
\draw [line width=1pt] (3.06,-0.98)-- (2.04,-0.98);
\draw [line width=1pt] (-0.94,-0.98)-- (0.04,-0.98);
\draw [line width=1pt] (-1.96,-0.98)-- (-0.94,-0.98);
\draw [line width=1pt] (-3.98,-0.98)-- (-2.98,-0.98);
\draw [line width=1pt,dotted] (3.06,-0.98)-- (4.02,-1);
\draw [line width=1pt,dotted] (1.02,-0.98)-- (2.04,-0.98);
\draw [line width=1pt,dotted] (0.02,0.54)-- (0.04,-0.3);
\draw (-3.94,-1.52) node[anchor=north west] {$q_2$};
\draw (2.98,-1.52) node[anchor=north west] {$q_1$};
\draw (0.62,2.42) node[anchor=north west] {$h$};
\draw [line width=0.8pt] (-6.001767441860463,-1.4293023255813937)-- (-0.7790697674418636,-1.38427906976744);
\draw [line width=0.8pt] (0.9093023255813963,-1.4293023255813937)-- (5.164,-1.38427906976744);
\draw [line width=0.8pt] (0.4140465116279048,2.982976744186044)-- (0.4140465116279048,-0.43879069767442);
\begin{scriptsize}
\draw [fill=ffffff] (-3.98,-0.98) circle (2.5pt);
\draw [fill=ffffff] (-4.98,-0.98) circle (2.5pt);
\draw [fill=ffffff] (5.68,-1.24) circle (2.5pt);
\draw [fill=ffffff] (6.18,-1.44) circle (2.5pt);
\draw [fill=ffffff] (0.04,-0.98) circle (2.5pt);
\draw [fill=ffffff] (5.62,-0.62) circle (2.5pt);
\draw [fill=ffffff] (6.14,-0.36) circle (2.5pt);
\draw [fill=ffffff] (1.02,-0.98) circle (2.5pt);
\draw [fill=ffffff] (4.02,-1) circle (2.5pt);
\draw [fill=ffffff] (-5.98,-0.98) circle (2.5pt);
\draw [fill=ffffff] (5.02,-0.96) circle (2.5pt);
\draw [fill=ffffff] (-6.48,-1.44) circle (2.5pt);
\draw [fill=ffffff] (-6.94,-1.88) circle (2.5pt);
\draw [fill=ffffff] (0.02,2.94) circle (2.5pt);
\draw [fill=ffffff] (0.02,2.06) circle (2.5pt);
\draw [fill=ffffff] (0.04,1.3) circle (2.5pt);
\draw [fill=ffffff] (0.02,0.54) circle (2.5pt);
\draw [fill=ffffff] (0.04,-0.3) circle (2.5pt);
\draw [fill=ffffff] (-2.98,-0.98) circle (2.5pt);
\draw [fill=ffffff] (-1.96,-0.98) circle (2.5pt);
\draw [fill=ffffff] (-7.4,-1.02) circle (2.5pt);
\draw [fill=ffffff] (-6.66,-1) circle (2.5pt);
\draw [fill=ffffff] (-7.06,-0.12) circle (2.5pt);
\draw [fill=ffffff] (-6.52,-0.54) circle (2.5pt);
\draw [fill=ffffff] (3.06,-0.98) circle (2.5pt);
\draw [fill=ffffff] (2.04,-0.98) circle (2.5pt);
\draw [fill=ffffff] (-1.96,-0.98) circle (2.5pt);
\draw [fill=ffffff] (-0.94,-0.98) circle (2.5pt);
\draw [fill=black,shift={(-6.001767441860463,-1.4293023255813937)},rotate=270] (0,0) ++(0 pt,2.25pt) -- ++(1.9485571585149868pt,-3.375pt)--++(-3.8971143170299736pt,0 pt) -- ++(1.9485571585149868pt,3.375pt);
\draw [fill=black,shift={(-0.7790697674418636,-1.38427906976744)},rotate=90] (0,0) ++(0 pt,2.25pt) -- ++(1.9485571585149868pt,-3.375pt)--++(-3.8971143170299736pt,0 pt) -- ++(1.9485571585149868pt,3.375pt);
\draw [fill=black,shift={(0.9093023255813963,-1.4293023255813937)},rotate=270] (0,0) ++(0 pt,2.25pt) -- ++(1.9485571585149868pt,-3.375pt)--++(-3.8971143170299736pt,0 pt) -- ++(1.9485571585149868pt,3.375pt);
\draw [fill=black,shift={(5.164,-1.38427906976744)},rotate=90] (0,0) ++(0 pt,2.25pt) -- ++(1.9485571585149868pt,-3.375pt)--++(-3.8971143170299736pt,0 pt) -- ++(1.9485571585149868pt,3.375pt);
\draw [fill=black,shift={(0.4140465116279048,2.982976744186044)},rotate=180] (0,0) ++(0 pt,2.25pt) -- ++(1.9485571585149868pt,-3.375pt)--++(-3.8971143170299736pt,0 pt) -- ++(1.9485571585149868pt,3.375pt);
\draw [fill=black,shift={(0.4140465116279048,-0.43879069767442)}] (0,0) ++(0 pt,2.25pt) -- ++(1.9485571585149868pt,-3.375pt)--++(-3.8971143170299736pt,0 pt) -- ++(1.9485571585149868pt,3.375pt);
\end{scriptsize}
\end{tikzpicture} 
  \caption{The family of trees $\mathfrak{F}(n)$. }\label{main_fam}
\end{figure}

The remaining of the paper is as follows. In order to explain that the family we study is not arbitrary, we devote the rest of this introduction to justify our choice. In Section \ref{sec:family}, we define the family $\mathfrak{F}(n)$ and the 2-switch transformations we perform. Moreover, we explain our powerful technique to obtain the order, that is based on an algorithmic tool, allowing to compare indices of trees without computing them. In Section \ref{sec:analytical properties} we obtain necessary analytical properties of some recurrence relations that appear in our comparison method. In Section \ref{sec:ordering the family} we show how the spectral radius varies upon 2-switches transformations. In Section \ref{sec: ordering}, we use the 2-switches to obtain a total ordering in $\mathfrak{F}(n)$. Finally, in Section \ref{sec: additional considerations} we reason how 2-switches may be useful for the class of problems proposed here.

\subsection{Motivation for choosing the family}

We start by introducing some notation from \cite{JaOlTr} which is specially useful to represent the trees in the family $\mathfrak{F}(n)$. In that paper, it was proven that the number of Laplacian eigenvalues less than the average degree $\avgdeg{n}$ of a tree having $n$ vertices is at least $\ceilof$. We remark that pendant paths of length 2 play an important role there and serve as a motivation for our choice.

Let $T$ be a tree with $n$ vertices, and let $u$ be a vertex of degree at least $\ell$ of $T$ having $\ell \geq 1$ pendant paths attached at $u$.
We denote the \emph{sum} of pendant paths attached at $u$ by $P(u)=P_{q_1}\oplus\cdots \oplus P_{q_\ell}$, as illustrated in Figure \ref{fig:ex1}. The number of edges in each path is denoted by $\sharp P_{q}=q$.
\begin{figure}[!ht]
  \centering
\definecolor{rvwvcq}{rgb}{0.08235294117647059,0.396078431372549,0.7529411764705882}
\begin{tikzpicture}[line cap=round,line join=round,>=triangle 45,x=1cm,y=1cm, scale=0.6, every node/.style={scale=0.7}]
\clip(-4,-4) rectangle (3,1);
\draw [line width=1pt,dash pattern=on 1pt off 1pt] (2,0.47)-- (1.34,-0.27);
\draw [line width=1pt] (1.34,-0.27)-- (0.86,-0.81);
\draw [line width=1pt] (0.86,-0.81)-- (0.08,-0.99);
\draw [line width=1pt] (0.08,-0.99)-- (-0.58,-1.17);
\draw [line width=1pt] (0.86,-0.81)-- (0.54,-1.27);
\draw [line width=1pt] (0.54,-1.27)-- (0.18,-1.75);
\draw [line width=1pt] (0.18,-1.75)-- (-0.2,-2.23);
\draw [line width=1pt] (-0.2,-2.23)-- (-0.56,-2.71);
\draw [line width=1pt] (0.86,-0.81)-- (1.22,-1.47);
\draw [line width=1pt] (0.86,-0.81)-- (0.7,-1.69);
\draw [line width=1pt] (0.7,-1.69)-- (0.58,-2.29);
\draw [line width=1pt] (0.58,-2.29)-- (0.48,-2.79);
\draw [line width=1pt] (0.48,-2.79)-- (0.38,-3.37);
\draw [line width=1pt] (0.86,-0.81)-- (0.22,-0.39);
\draw [line width=1pt] (0.22,-0.39)-- (-0.38,-0.01);
\draw (0.42,0.19) node[anchor=north west] {\tiny{$u$}};
\draw (-1.6,0.49) node[anchor=north west] {\tiny{$P_2$}};
\draw (-1.7,-0.85) node[anchor=north west] {\tiny{$P_2$}};
\draw (0.6,-2.83) node[anchor=north west] {\tiny{$P_4$}};
\draw (-1.8,-2.25) node[anchor=north west] {\tiny{$P_5$}};
\draw (1.3,-1.47) node[anchor=north west] {\tiny{$P_1$}};
\draw [line width=1pt] (-0.56,-2.71)-- (-0.9,-3.21);
\begin{scriptsize}
\draw [fill=black] (1.34,-0.27) circle (2.5pt);
\draw [fill=black] (0.86,-0.81) circle (2.5pt);
\draw [fill=black] (0.08,-0.99) circle (2.5pt);
\draw [fill=black] (-0.58,-1.17) circle (2.5pt);
\draw [fill=black] (0.54,-1.27) circle (2.5pt);
\draw [fill=black] (0.18,-1.75) circle (2.5pt);
\draw [fill=black] (-0.2,-2.23) circle (2.5pt);
\draw [fill=black] (-0.56,-2.71) circle (2.5pt);
\draw [fill=black] (1.22,-1.47) circle (2.5pt);
\draw [fill=black] (0.7,-1.69) circle (2.5pt);
\draw [fill=black] (0.58,-2.29) circle (2.5pt);
\draw [fill=black] (0.48,-2.79) circle (2.5pt);
\draw [fill=black] (0.38,-3.37) circle (2.5pt);
\draw [fill=black] (0.22,-0.39) circle (2.5pt);
\draw [fill=black] (-0.38,-0.01) circle (2.5pt);
\draw [fill=black] (-0.9,-3.21) circle (2.5pt);
\end{scriptsize}
\end{tikzpicture}
\qquad
\begin{tikzpicture}[line cap=round,line join=round,>=triangle 45,x=1cm,y=1cm,scale=0.6, every node/.style={scale=0.7}]
\clip(-8.86,-0.64) rectangle (2.58,4);
\draw [line width=1pt,dash pattern=on 1pt off 1pt] (-6.12,0.44)-- (-6.92,0);
\draw [line width=1pt] (-4.38526,1.46298)-- (-3.8329483896370125,2.3307297959890714);
\draw [line width=1pt] (-3.8329483896370125,2.3307297959890714)-- (-3.5034939745135114,2.8484438768974316);
\draw [line width=1pt] (-4.38526,1.46298)-- (-4.068272971868086,2.4248596288815003);
\draw [line width=1pt] (-4.068272971868086,2.4248596288815003)-- (-3.8643250006011565,3.0994567646105757);
\draw [line width=1pt] (-4.38526,1.46298)-- (-4.31928585958123,2.4876128508097866);
\draw [line width=1pt] (-4.31928585958123,2.4876128508097866)-- (-4.3035975540991585,3.1465216810567904);
\draw [line width=1pt] (-4.38526,1.46298)-- (-4.63305196922266,2.4876128508097866);
\draw [line width=1pt] (-4.63305196922266,2.4876128508097866)-- (-4.758558413079232,3.162209986538862);
\draw (-5.28,1.04) node[anchor=north west] {{\tiny $P_3$}};
\draw [line width=1pt,dash pattern=on 1pt off 1pt] (0.44,1.13)-- (-0.36,0.69);
\draw [line width=1pt] (0.44,1.13)-- (1.25474,1.57298);
\draw (0.64,1.30) node[anchor=north west] {{\tiny $P_3*S_4$}};
\draw [->,line width=1pt] (-3.14,0.34) -- (-1.82,0.34);
\draw (-4.26,3.9) node[anchor=north west] {{\tiny $4\, P_2$}};
\draw (-4.34,1.66) node[anchor=north west] {{\tiny $v$}};
\draw (-4.6,0.3) node[anchor=north west] {{\tiny $(P_q,S_r) \text{ representation}$}};
\draw (0.08,1.74) node[anchor=north west] {{\tiny $u$}};
\draw [line width=1pt] (-6.12,0.44)-- (-5.58,0.76);
\draw [line width=1pt] (-5.58,0.76)-- (-5.02,1.08);
\draw [line width=1pt] (-5.02,1.08)-- (-4.38526,1.46298);
\draw (-6.5,1.08) node[anchor=north west] {{\tiny $u$}};
\draw [line width=1pt,dash pattern=on 1pt off 1pt] (-6.12,0.44)-- (-6.34,-0.38);
\draw [line width=1pt,dash pattern=on 1pt off 1pt] (0.44,1.13)-- (0.36,0.32);
\begin{scriptsize}
\draw [fill=black] (-6.12,0.44) circle (2.5pt);
\draw [fill=black] (-4.38526,1.46298) circle (3pt);
\draw [fill=black] (-3.8329483896370125,2.3307297959890714) circle (2.5pt);
\draw [fill=black] (-3.5034939745135114,2.8484438768974316) circle (2.5pt);
\draw [fill=black] (-4.068272971868086,2.4248596288815003) circle (2.5pt);
\draw [fill=black] (-3.8643250006011565,3.0994567646105757) circle (2.5pt);
\draw [fill=black] (-4.31928585958123,2.4876128508097866) circle (2.5pt);
\draw [fill=black] (-4.3035975540991585,3.1465216810567904) circle (2.5pt);
\draw [fill=black] (-4.63305196922266,2.4876128508097866) circle (2.5pt);
\draw [fill=black] (-4.758558413079232,3.162209986538862) circle (2.5pt);
\draw [fill=black] (0.44,1.13) circle (2.5pt);
\draw [fill=black] (1.25474,1.57298) ++(-4.5pt,0 pt) -- ++(4.5pt,4.5pt)--++(4.5pt,-4.5pt)--++(-4.5pt,-4.5pt)--++(-4.5pt,4.5pt);
\draw [fill=black] (-5.58,0.76) circle (2.5pt);
\draw [fill=black] (-5.02,1.08) circle (2.5pt);
\end{scriptsize}
\end{tikzpicture} 
\caption{Vertex $u$ with $P(u)=P_{1}\oplus 2P_{2}\oplus P_{4}\oplus P_{5}$ (left) and $(P_q, S_r)$ representation of a generalized pendant path (right).}\label{fig:ex1}
\end{figure}
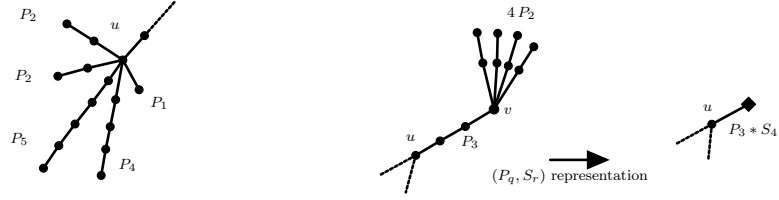
A subgraph obtained by a vertex $u$ attached to $r\geq 1$ paths of length 2,
is called a {\em sun with $r$ rays} and denoted by $S_r$.
To simplify the representation, we use the concatenation symbol and write $P_q \ast S_r$.
To further simplify the graphic part of the representation, we will use a black square {\tiny{$\blacksquare$}} to represent a pendant sun $S_r$ attached to a vertex, and a single edge to represent the entire path $P_q$, omitting the $r$ pendant $P_2$'s and the $q$ vertices.
We will refer to this as the $(P_q,S_r)$ \emph {representation} of this generalized pendant path $P_q \ast S_r$, as shown in Figure~\ref{fig:ex1} (right).
We can consider $q=0$ for paths
$P_q$ of length 0, as well as $r=0$ for no pendant $S_r$.
However, we do \emph{not} allow both $r=q=0$ simultaneously.\\
We say a vertex $u$ is a starlike vertex if it has degree $\geq 3$ and has at least two generalized pendant paths attached to it.
Using this notation, we can write any tree $T$ having a single starlike vertex $u$ as
$T=u + P_{q_{1}} \ast  S_{r_{1}} \oplus \cdots \oplus P_{q_{\ell}} \ast  S_{r_{\ell}} $ where $\ell \geq 1$.

In particular, any  member of $\mathfrak{F}(n)$, as in the Figure~\ref{main_fam}, has a single starlike vertex $u$ and is represented as
$$T=u + P_{h} \ast  S_{0} \oplus P_{q_{1}} \ast  S_{r_{1}} \oplus P_{q_{2}} \ast  S_{r_{2}},$$
for $h, q_1, q_2  \geq 2$.

Our choice of the family is related to the fact that we would like to study the spectral radius ordering of 2-switches on trees having a single starlike vertex. We notice that the general family of trees $\mathfrak{H}$ with just one starlike vertex in this notation is given by,
$$\mathfrak{H}:=\{T\,|\,T=u + P_{q_0} \ast  S_{r_0} \oplus P_{q_{1}} \ast  S_{r_{1}} \oplus P_{q_{2}} \ast  S_{r_{2}},\; q_0, q_1, q_2  \geq 0, r_0, r_1, r_2 \geq 0\}.$$
In this generality, our tool becames too involved and, hence, in order to simplify the computations and to obtain symmetry, we have chosen to study a special case of this family $\mathfrak{H}$, where we replace $S_{r_0}$ by $S_0$ (no pendant $P_2$'s on the path $P_h$).

\section{A family and our tool}
\label{sec:family}
In this section we define the family which we will determine a spectral radius ordering when performing 2-switch operation, as well as an algorithmic tool that we believe it is powerful for this class of problems.

\subsection{Our Family}
We consider the family $\mathfrak{F}(n)$ of all trees $T$ given by Figure~\ref{main_fam}, with the following constraints:
\begin{itemize}
  \item[1)] $q_1, q_2 \geq 2$;% and $q_1 \neq q_2$;
  \item[2)] $r_1, r_2 \geq 2$ and $r_1 < r_2$;
  \item[3)] $h \geq 2$.
\end{itemize}

Notice that, in this case, the number of vertices is $n=|T|= 1 + h + q_1 + q_2 + 2 (r_1+r_2) \geq 7+ 2 (r_1+r_2)\geq 17$.

We are interested in ordering by the index (spectral radius) the trees with a fixed degree sequence, in  $\mathfrak{F}(n)$. As a particular case, we obtain the extreme members of the family, that is, the trees in $\mathfrak{F}(n)$ having largest and minimum spectral radius. The degree sequence a $T \in \mathfrak{F}(n)$ is given by
$$d:=[r_1 +1, r_2 +1, 3, 2^{n-(r_1+r_2) -2},1^{r_1+r_2+1}].$$
In order to keep the degree sequence we fix $r_2 > r_1 \geq 2$. In this way each element in $\mathfrak{F}(n)$, with $n$ vertices and degree sequence $d$ is uniquely determined by the 3-uple $[h, q_1, q_2]$ that is,
$$\mathfrak{F}(n):=\{T=[h, q_1, q_2] \; | \; h + q_1 + q_2 =n-1-2 (r_1+r_2)\}.$$

Now we consider two types of 2-switches on $T=[h, q_1, q_2] \in \mathfrak{F}(n)$. Later, in Section \ref{sec:ordering the family}, we determine how to order these operations by their spectral radii.

\begin{itemize}
 \item[Type I -] We switch the vertices between the central path and the right (or the left) branch of  $T$ obtaining a new tree $T$(see Figure~\ref{TypeISwitch}).  More precisely, we have a  new member of $\mathfrak{F}(n)$, where the parameters are changed by $q'_1 =h-s+t$ (the closest to $r_1 \, P_2$) and $h'=q_1-t+s$ (the new central path after some relabeling). Let us call that a $(s,t)$-2-switch of type II (see Figure~\ref{TypeISwitch}). In order to do the 2-switch, we disconnect the edges $[u_1,u_2]$ and $[w_1,w_2]$ and reconnect $[u_1,w_2]$ and $[w_1,u_2]$.
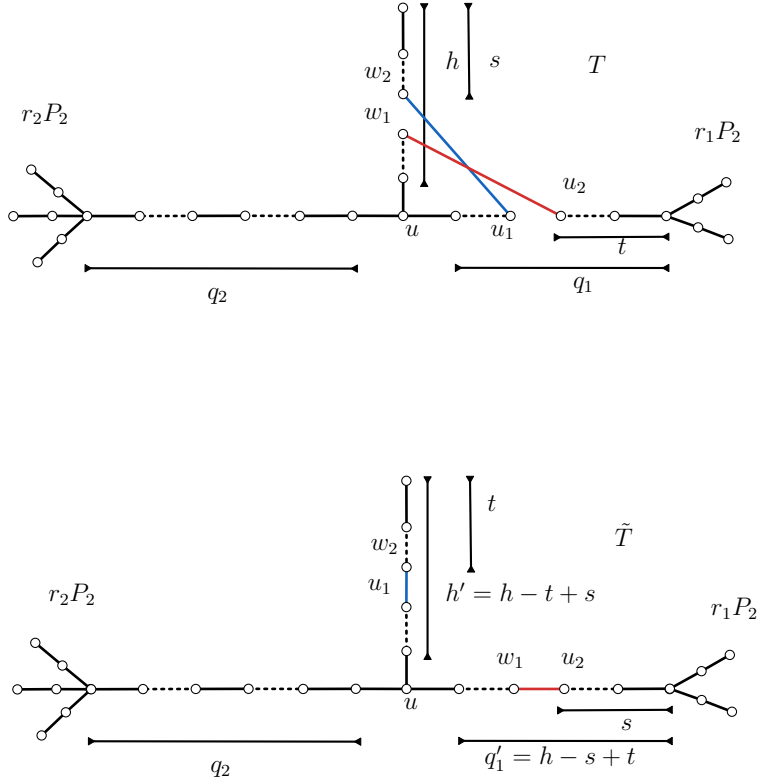
\begin{figure}[H]
  \centering
\definecolor{dtsfsf}{rgb}{0.8274509803921568,0.1843137254901961,0.1843137254901961}
\definecolor{rvwvcq}{rgb}{0.08235294117647059,0.396078431372549,0.7529411764705882}
\definecolor{ffffff}{rgb}{1,1,1}
\begin{tikzpicture}[line cap=round,line join=round,>=triangle 45,x=1cm,y=1cm,scale=0.7, every node/.style={scale=0.8}]
\clip(-11.0,-12.8) rectangle (11.0,3.8);
\draw [line width=1pt] (5.604841723994303,-1.217642963465425)-- (6.18,-1.44);
\draw [line width=1pt] (5.602336040821398,-0.6688983485991646)-- (6.14,-0.36);
\draw [line width=1pt] (0.0030355131150047,-0.9993162365969024)-- (0.9993960521943549,-0.9996485274226641);
\draw [line width=1pt] (-5.998977049730127,-1.0021542105955705)-- (-5.001715146913821,-1.0021542105955705);
\draw [line width=1pt] (-6.48,-1.44)-- (-6.94,-1.88);
\draw [line width=1pt] (4.02,-1)-- (5.000972079324126,-1.0021542105955694);
\draw [line width=1pt] (5.000972079324126,-1.0021542105955694)-- (5.604841723994303,-1.217642963465425);
\draw [line width=1pt] (5.000972079324126,-1.0021542105955694)-- (5.602336040821398,-0.6688983485991646);
\draw [line width=1pt,dotted] (-3.999441877751704,-0.999648527422664)-- (-5.001715146913821,-1.0021542105955705);
\draw (-0.1239266590351313,-1.0001214931835896) node[anchor=north west] {$u$};
\draw [line width=1pt] (-5.998977049730127,-1.0021542105955705)-- (-6.48,-1.44);
\draw [line width=1pt] (0.000008795653280872282,2.9566034858761667)-- (0.0030355131150047,2.075828704514527);
\draw [line width=1pt,dotted] (-2.9971686085895866,-0.9971428442497587)-- (-1.9598157750067955,-0.9996485274226641);
\draw [line width=1pt] (-7.399653943384187,-1.0046598937684774)-- (-6.66,-1);
\draw [line width=1pt] (-7.06,-0.12)-- (-6.545215981423482,-0.5511312394726173);
\draw [line width=1pt] (-6.545215981423482,-0.5511312394726173)-- (-5.998977049730127,-1.0021542105955705);
\draw [line width=1pt] (-6.66,-1)-- (-5.998977049730127,-1.0021542105955705);
\draw (5.3692032803048635,0.9637436146633477) node[anchor=north west] {$r_1 P_2$};
\draw (-7.410150827278855,1.3052853725497715) node[anchor=north west] {$r_2 P_2$};
\draw [line width=1pt,dotted] (0.0030355131150047,2.075828704514527)-- (0.0030355131150047,1.3161226216218413);
\draw [line width=1pt] (0.0030355131150047,-0.27895748070662657)-- (0.0030355131150047,-0.9993162365969024);
\draw [line width=1pt] (-0.9600481890175837,-0.9971428442497587)-- (0.0030355131150047,-0.9993162365969024);
\draw [line width=1pt,dotted] (2.9989312241727784,-0.9996485274226641)-- (4.02,-1);
\draw [line width=1pt,dotted] (0.9993960521943549,-0.9996485274226641)-- (2.0417602521229568,-0.9996485274226641);
\draw [line width=1pt,dotted] (0.000008795653280872282,0.5564165387291554)-- (0.0030355131150047,-0.27895748070662657);
\draw (-3.8808859957858015,-2.2) node[anchor=north west] {$q_2$};
\draw (3.063796414571498,-2) node[anchor=north west] {$q_1$};
\draw (0.6445422962093239,2.3014488330518414) node[anchor=north west] {$h$};
\draw [line width=0.8pt] (-5.998722801439957,-1.993969849943333)-- (-0.9055039114889535,-2.0012627655088457);
\draw [line width=0.8pt] (1.034112644044173,-1.9864528004246156)-- (5.009814577362994,-1.9896298874580076);
\draw [line width=0.8pt] (0.40256221806254994,2.992924095416853)-- (0.40256221806254994,-0.4000261791755805);
\draw [line width=1pt] (-1.9598157750067955,-0.9996485274226641)-- (-0.9600481890175837,-0.9971428442497587);
\draw (1.4983966909253854, -1) node[anchor=north west] {$u_1$};
\draw (-0.9,1.2198999330781657) node[anchor=north west] {$w_1$};
\draw (-0.9,1.9883688883226194) node[anchor=north west] {$w_2$};
\draw (2.8645637224710834,-0.11780528531032793) node[anchor=north west] {$u_2$};
\draw [line width=0.8pt] (2.9108594463914534,-1.4067713010116925)-- (5.009814577362994,-1.396353106865182);
\draw (1.4699348777681833,2.2014488330518414) node[anchor=north west] {$s$};
\draw (3.9176508092875593,-1.2570451194979935) node[anchor=north west] {$t$};
\draw [line width=1pt] (0.06693641282295731,-9.99028245728597)-- (1.0632969519023077,-9.99061474811173);
\draw [line width=1pt,dotted] (-3.9355409780437496,-9.99061474811173)-- (-4.937814247205868,-9.993120431284636);
\draw (-0.15238847219233334,-9.994054450859418) node[anchor=north west] {$u$};
\draw [line width=1pt] (0.0639096953612342,-6.034362734812899)-- (0.06693641282295731,-6.915137516174538);
\draw [line width=1pt,color=rvwvcq] (0.06693641282295731,-7.674843599067225)-- (0.0639096953612342,-8.434549681959911);
\draw [line width=1pt,dotted] (-2.9332677088816332,-9.988109064938824)-- (-1.8959148752988424,-9.99061474811173);
\draw (5.682283225034086,-8.087112969326885) node[anchor=north west] {$r_1 P_2$};
\draw (-6.897838190449218,-7.830956650912067) node[anchor=north west] {$r_2 P_2$};
\draw [line width=1pt,dotted] (0.06693641282295731,-6.915137516174538)-- (0.06693641282295731,-7.674843599067225);
\draw [line width=1pt] (0.06693641282295731,-9.269923701395692)-- (0.06693641282295731,-9.99028245728597);
\draw [line width=1pt] (-0.8961472893096307,-9.988109064938824)-- (0.06693641282295731,-9.99028245728597);
\draw [line width=1pt,dotted] (3.0628321238807334,-9.99061474811173)-- (4.083900899707954,-9.990966220689065);
\draw [line width=1pt,dotted] (1.0632969519023077,-9.99061474811173)-- (2.105661151830911,-9.99061474811173);
\draw [line width=1pt,dotted] (0.0639096953612342,-8.434549681959911)-- (0.06693641282295731,-9.269923701395692);
\draw (-3.8239623694713973,-11.2) node[anchor=north west] {$q_2$};
\draw (1.384549438296577,-10.904832471889883) node[anchor=north west] {$q'_1=h-s+t$};
\draw (0.6445422962093239,-7.802494837754865) node[anchor=north west] {$h'=h-t+s$};
\draw [line width=0.8pt] (-5.934821901732002,-10.984936070632399)-- (-0.8416030117810005,-10.99222898619791);
\draw [line width=0.8pt] (1.0980135437521261,-10.97741902111368)-- (5.073715477070952,-10.980596108147074);
\draw [line width=0.8pt] (0.4664631177705029,-5.998042125272213)-- (0.4664631177705029,-9.390992399864645);
\draw [line width=1pt] (-1.8959148752988424,-9.99061474811173)-- (-0.8961472893096307,-9.988109064938824);
\draw (1.6122439435541935,-9.15) node[anchor=north west] {$w_1$};
\draw (2.8361019093138813,-9.15) node[anchor=north west] {$u_2$};
\draw [line width=0.8pt] (2.9747603460994085,-10.39773752170076)-- (5.073715477070952,-10.387319327554247);
\draw (1.4414730646109812,-6.15170967463715) node[anchor=north west] {$t$};
\draw (3.9745744356019634,-10.38978077173822) node[anchor=north west] {$s$};
\draw [line width=1pt] (-3.9355409780437496,-9.99061474811173)-- (-2.9332677088816332,-9.988109064938824);
\draw [line width=1pt,color=dtsfsf] (2.105661151830911,-9.99061474811173)-- (3.0628321238807334,-9.99061474811173);
\draw (3.3484145461435184,2.216063393580235) node[anchor=north west] {$T$};
\draw (3.860727182973155,-6.664022311466786) node[anchor=north west] {$\tilde{T}$};
\draw [line width=0.8pt] (1.24,3)-- (1.26,1.26);
\draw [line width=1pt] (-3.999441877751704,-0.999648527422664)-- (-2.9971686085895866,-0.9971428442497587);
\draw [line width=1pt,color=rvwvcq] (0.0030355131150047,1.3161226216218413)-- (2.0417602521229568,-0.9996485274226641);
\draw [line width=1pt,color=dtsfsf] (0.000008795653280872282,0.5564165387291554)-- (2.9989312241727784,-0.9996485274226641);
\draw (-0.8354719879651824,-7.745571211440462) node[anchor=north west] {$u_1$};
\draw (-0.7785483616507783,-6.977102256196008) node[anchor=north west] {$w_2$};
\draw [line width=1pt] (-6.403293526854049,-10.434223647809214)-- (-6.863293526854047,-10.874223647809211);
\draw [line width=1pt] (-5.922270576584174,-9.996377858404783)-- (-6.403293526854049,-10.434223647809214);
\draw [line width=1pt] (-7.32294747023824,-9.998883541577689)-- (-6.583293526854048,-9.994223647809212);
\draw [line width=1pt] (-6.983293526854048,-9.114223647809212)-- (-6.468509508277529,-9.54535488728183);
\draw [line width=1pt] (-6.468509508277529,-9.54535488728183)-- (-5.922270576584174,-9.996377858404783);
\draw [line width=1pt] (-6.583293526854048,-9.994223647809212)-- (-5.922270576584174,-9.996377858404783);
\draw [line width=1pt] (-5.922270576584174,-9.996377858404783)-- (-4.937814247205868,-9.993120431284636);
\draw [line width=1pt] (5.671890322266495,-10.201419289091856)-- (6.24704859827219,-10.423776325626429);
\draw [line width=1pt] (5.66938463909359,-9.652674674225594)-- (6.207048598272195,-9.343776325626429);
\draw [line width=1pt] (5.068020677596319,-9.985930536221998)-- (5.671890322266495,-10.201419289091856);
\draw [line width=1pt] (5.068020677596319,-9.985930536221998)-- (5.66938463909359,-9.652674674225594);
\draw [line width=1pt] (4.083900899707954,-9.990966220689065)-- (5.068020677596319,-9.985930536221998);
\draw [line width=0.8pt] (1.2728196275257293,-5.988749894975701)-- (1.2928196275257293,-7.728749894975701);
\begin{scriptsize}
\draw [fill=ffffff] (-3.999441877751704,-0.999648527422664) circle (2.5pt);
\draw [fill=ffffff] (-5.001715146913821,-1.0021542105955705) circle (2.5pt);
\draw [fill=ffffff] (5.604841723994303,-1.217642963465425) circle (2.5pt);
\draw [fill=ffffff] (6.18,-1.44) circle (2.5pt);
\draw [fill=ffffff] (0.0030355131150047,-0.9993162365969024) circle (2.5pt);
\draw [fill=ffffff] (5.602336040821398,-0.6688983485991646) circle (2.5pt);
\draw [fill=ffffff] (6.14,-0.36) circle (2.5pt);
\draw [fill=ffffff] (0.9993960521943549,-0.9996485274226641) circle (2.5pt);
\draw [fill=ffffff] (4.02,-1) circle (2.5pt);
\draw [fill=ffffff] (-5.998977049730127,-1.0021542105955705) circle (2.5pt);
\draw [fill=ffffff] (5.000972079324126,-1.0021542105955694) circle (2.5pt);
\draw [fill=ffffff] (-6.48,-1.44) circle (2.5pt);
\draw [fill=ffffff] (-6.94,-1.88) circle (2.5pt);
\draw [fill=ffffff] (0.000008795653280872282,2.9566034858761667) circle (2.5pt);
\draw [fill=ffffff] (0.0030355131150047,2.075828704514527) circle (2.5pt);
\draw [fill=ffffff] (0.0030355131150047,1.3161226216218413) circle (2.5pt);
\draw [fill=ffffff] (0.000008795653280872282,0.5564165387291554) circle (2.5pt);
\draw [fill=ffffff] (0.0030355131150047,-0.27895748070662657) circle (2.5pt);
\draw [fill=ffffff] (-2.9971686085895866,-0.9971428442497587) circle (2.5pt);
\draw [fill=ffffff] (-1.9598157750067955,-0.9996485274226641) circle (2.5pt);
\draw [fill=ffffff] (-7.399653943384187,-1.0046598937684774) circle (2.5pt);
\draw [fill=ffffff] (-6.66,-1) circle (2.5pt);
\draw [fill=ffffff] (-7.06,-0.12) circle (2.5pt);
\draw [fill=ffffff] (-6.545215981423482,-0.5511312394726173) circle (2.5pt);
\draw [fill=ffffff] (2.9989312241727784,-0.9996485274226641) circle (2.5pt);
\draw [fill=ffffff] (2.0417602521229568,-0.9996485274226641) circle (2.5pt);
\draw [fill=ffffff] (-0.9600481890175837,-0.9971428442497587) circle (2.5pt);
\draw [fill=black,shift={(-5.998722801439957,-1.993969849943333)},rotate=270] (0,0) ++(0 pt,2.25pt) -- ++(1.9485571585149868pt,-3.375pt)--++(-3.8971143170299736pt,0 pt) -- ++(1.9485571585149868pt,3.375pt);
\draw [fill=black,shift={(-0.9055039114889535,-2.0012627655088457)},rotate=90] (0,0) ++(0 pt,2.25pt) -- ++(1.9485571585149868pt,-3.375pt)--++(-3.8971143170299736pt,0 pt) -- ++(1.9485571585149868pt,3.375pt);
\draw [fill=black,shift={(1.034112644044173,-1.9864528004246156)},rotate=270] (0,0) ++(0 pt,2.25pt) -- ++(1.9485571585149868pt,-3.375pt)--++(-3.8971143170299736pt,0 pt) -- ++(1.9485571585149868pt,3.375pt);
\draw [fill=black,shift={(5.009814577362994,-1.9896298874580076)},rotate=90] (0,0) ++(0 pt,2.25pt) -- ++(1.9485571585149868pt,-3.375pt)--++(-3.8971143170299736pt,0 pt) -- ++(1.9485571585149868pt,3.375pt);
\draw [fill=black,shift={(0.40256221806254994,2.992924095416853)},rotate=180] (0,0) ++(0 pt,2.25pt) -- ++(1.9485571585149868pt,-3.375pt)--++(-3.8971143170299736pt,0 pt) -- ++(1.9485571585149868pt,3.375pt);
\draw [fill=black,shift={(0.40256221806254994,-0.4000261791755805)}] (0,0) ++(0 pt,2.25pt) -- ++(1.9485571585149868pt,-3.375pt)--++(-3.8971143170299736pt,0 pt) -- ++(1.9485571585149868pt,3.375pt);
\draw [fill=black,shift={(2.9108594463914534,-1.4067713010116925)},rotate=270] (0,0) ++(0 pt,2.25pt) -- ++(1.9485571585149868pt,-3.375pt)--++(-3.8971143170299736pt,0 pt) -- ++(1.9485571585149868pt,3.375pt);
\draw [fill=black,shift={(5.009814577362994,-1.396353106865182)},rotate=90] (0,0) ++(0 pt,2.25pt) -- ++(1.9485571585149868pt,-3.375pt)--++(-3.8971143170299736pt,0 pt) -- ++(1.9485571585149868pt,3.375pt);
\draw [fill=ffffff] (-3.9355409780437496,-9.99061474811173) circle (2.5pt);
\draw [fill=ffffff] (-4.937814247205868,-9.993120431284636) circle (2.5pt);
\draw [fill=ffffff] (0.06693641282295731,-9.99028245728597) circle (2.5pt);
\draw [fill=ffffff] (1.0632969519023077,-9.99061474811173) circle (2.5pt);
\draw [fill=ffffff] (4.083900899707954,-9.990966220689065) circle (2.5pt);
\draw [fill=ffffff] (-5.935076150022173,-9.993120431284636) circle (2.5pt);
\draw [fill=ffffff] (5.0648729790320814,-9.993120431284636) circle (2.5pt);
\draw [fill=ffffff] (0.0639096953612342,-6.034362734812899) circle (2.5pt);
\draw [fill=ffffff] (0.06693641282295731,-6.915137516174538) circle (2.5pt);
\draw [fill=ffffff] (0.06693641282295731,-7.674843599067225) circle (2.5pt);
\draw [fill=ffffff] (0.0639096953612342,-8.434549681959911) circle (2.5pt);
\draw [fill=ffffff] (0.06693641282295731,-9.269923701395692) circle (2.5pt);
\draw [fill=ffffff] (-2.9332677088816332,-9.988109064938824) circle (2.5pt);
\draw [fill=ffffff] (-1.8959148752988424,-9.99061474811173) circle (2.5pt);
\draw [fill=ffffff] (3.0628321238807334,-9.99061474811173) circle (2.5pt);
\draw [fill=ffffff] (2.105661151830911,-9.99061474811173) circle (2.5pt);
\draw [fill=ffffff] (-0.8961472893096307,-9.988109064938824) circle (2.5pt);
\draw [fill=black,shift={(-5.934821901732002,-10.984936070632399)},rotate=270] (0,0) ++(0 pt,2.25pt) -- ++(1.9485571585149868pt,-3.375pt)--++(-3.8971143170299736pt,0 pt) -- ++(1.9485571585149868pt,3.375pt);
\draw [fill=black,shift={(-0.8416030117810005,-10.99222898619791)},rotate=90] (0,0) ++(0 pt,2.25pt) -- ++(1.9485571585149868pt,-3.375pt)--++(-3.8971143170299736pt,0 pt) -- ++(1.9485571585149868pt,3.375pt);
\draw [fill=black,shift={(1.0980135437521261,-10.97741902111368)},rotate=270] (0,0) ++(0 pt,2.25pt) -- ++(1.9485571585149868pt,-3.375pt)--++(-3.8971143170299736pt,0 pt) -- ++(1.9485571585149868pt,3.375pt);
\draw [fill=black,shift={(5.073715477070952,-10.980596108147074)},rotate=90] (0,0) ++(0 pt,2.25pt) -- ++(1.9485571585149868pt,-3.375pt)--++(-3.8971143170299736pt,0 pt) -- ++(1.9485571585149868pt,3.375pt);
\draw [fill=black,shift={(0.4664631177705029,-5.998042125272213)},rotate=180] (0,0) ++(0 pt,2.25pt) -- ++(1.9485571585149868pt,-3.375pt)--++(-3.8971143170299736pt,0 pt) -- ++(1.9485571585149868pt,3.375pt);
\draw [fill=black,shift={(0.4664631177705029,-9.390992399864645)}] (0,0) ++(0 pt,2.25pt) -- ++(1.9485571585149868pt,-3.375pt)--++(-3.8971143170299736pt,0 pt) -- ++(1.9485571585149868pt,3.375pt);
\draw [fill=black,shift={(2.9747603460994085,-10.39773752170076)},rotate=270] (0,0) ++(0 pt,2.25pt) -- ++(1.9485571585149868pt,-3.375pt)--++(-3.8971143170299736pt,0 pt) -- ++(1.9485571585149868pt,3.375pt);
\draw [fill=black,shift={(5.073715477070952,-10.387319327554247)},rotate=90] (0,0) ++(0 pt,2.25pt) -- ++(1.9485571585149868pt,-3.375pt)--++(-3.8971143170299736pt,0 pt) -- ++(1.9485571585149868pt,3.375pt);
\draw [fill=black,shift={(1.24,3)},rotate=180] (0,0) ++(0 pt,2.25pt) -- ++(1.9485571585149868pt,-3.375pt)--++(-3.8971143170299736pt,0 pt) -- ++(1.9485571585149868pt,3.375pt);
\draw [fill=black,shift={(1.26,1.26)}] (0,0) ++(0 pt,2.25pt) -- ++(1.9485571585149868pt,-3.375pt)--++(-3.8971143170299736pt,0 pt) -- ++(1.9485571585149868pt,3.375pt);
\draw [fill=ffffff] (-5.922270576584174,-9.996377858404783) circle (2.5pt);
\draw [fill=ffffff] (-6.403293526854049,-10.434223647809214) circle (2.5pt);
\draw [fill=ffffff] (-6.863293526854047,-10.874223647809211) circle (2.5pt);
\draw [fill=ffffff] (-7.32294747023824,-9.998883541577689) circle (2.5pt);
\draw [fill=ffffff] (-6.583293526854048,-9.994223647809212) circle (2.5pt);
\draw [fill=ffffff] (-6.983293526854048,-9.114223647809212) circle (2.5pt);
\draw [fill=ffffff] (-6.468509508277529,-9.54535488728183) circle (2.5pt);
\draw [fill=ffffff] (5.671890322266495,-10.201419289091856) circle (2.5pt);
\draw [fill=ffffff] (6.24704859827219,-10.423776325626429) circle (2.5pt);
\draw [fill=ffffff] (5.66938463909359,-9.652674674225594) circle (2.5pt);
\draw [fill=ffffff] (6.207048598272195,-9.343776325626429) circle (2.5pt);
\draw [fill=ffffff] (5.068020677596319,-9.985930536221998) circle (2.5pt);
\draw [fill=black,shift={(1.2728196275257293,-5.988749894975701)},rotate=180] (0,0) ++(0 pt,2.25pt) -- ++(1.9485571585149868pt,-3.375pt)--++(-3.8971143170299736pt,0 pt) -- ++(1.9485571585149868pt,3.375pt);
\draw [fill=black,shift={(1.2928196275257293,-7.728749894975701)}] (0,0) ++(0 pt,2.25pt) -- ++(1.9485571585149868pt,-3.375pt)--++(-3.8971143170299736pt,0 pt) -- ++(1.9485571585149868pt,3.375pt);
\end{scriptsize}
\end{tikzpicture} 
  \caption{The 2-switch of type I.}\label{TypeISwitch}
\end{figure}
  \item[Type \; II -] We switch the vertices between the right and the left branches of  $T$ obtaining a new tree $T'$; Notice that this transformation is indeed a 2-switch preserving the degree sequence ($d:=[r_1 +1, r_2 +1, 3, 2^{n-(r_1+r_2) -2},1^{r_1+r_2+1}]$). More precisely, we have a new member of $\mathfrak{F}(n)$, where the parameters are changed by $q'_1 =q_2+t-s$ (the closest to $r_1 \, P_2$) and $q'_2=q_1-t+s$ (the closest to $r_1 \, P_2$). Let us call that a $(s,t)$-2-switch of type II (see Figure~\ref{TypeIISwitch}). In order to do the 2-switch, we disconnect the edges $[u_1,u_2]$ and $[v_1,v_2]$ and reconnect $[u_1,v_2]$ and $[v_1,u_2]$.
\end{itemize}

\begin{figure}[H]
  \centering
\definecolor{rvwvcq}{rgb}{0.08235294117647059,0.396078431372549,0.7529411764705882}
\definecolor{dtsfsf}{rgb}{0.8274509803921568,0.1843137254901961,0.1843137254901961}
\definecolor{ffffff}{rgb}{1,1,1}
\begin{tikzpicture}[line cap=round,line join=round,>=triangle 45,x=1cm,y=1cm,scale=0.7, every node/.style={scale=0.8}]
\clip(-11.762222567061546,-13.104218056985475) rectangle (11.557100760720848,4.399421051165193);
\draw [line width=1pt] (5.604841723994303,-1.217642963465425)-- (6.18,-1.44);
\draw [line width=1pt] (5.602336040821398,-0.6688983485991646)-- (6.14,-0.36);
\draw [line width=1pt] (0.0030355131150047,-0.9993162365969024)-- (0.9993960521943549,-0.9996485274226641);
\draw [line width=1pt] (-5.998977049730127,-1.0021542105955705)-- (-5.001715146913821,-1.0021542105955705);
\draw [line width=1pt] (-6.48,-1.44)-- (-6.94,-1.88);
\draw [line width=1pt] (4.02,-1)-- (5.000972079324126,-1.0021542105955694);
\draw [line width=1pt] (5.000972079324126,-1.0021542105955694)-- (5.604841723994303,-1.217642963465425);
\draw [line width=1pt] (5.000972079324126,-1.0021542105955694)-- (5.602336040821398,-0.6688983485991646);
\draw [line width=1pt,dotted] (-3.999441877751704,-0.999648527422664)-- (-5.001715146913821,-1.0021542105955705);
\draw (-0.14605313300968167,-1.004621307347788) node[anchor=north west] {$u$};
\draw [line width=1pt] (-5.998977049730127,-1.0021542105955705)-- (-6.48,-1.44);
\draw [line width=1pt] (0.000008795653280872282,2.9566034858761667)-- (0.0030355131150047,2.075828704514527);
\draw [line width=1pt] (0.0030355131150047,1.3161226216218413)-- (0.000008795653280872282,0.5564165387291554);
\draw [line width=1pt,dotted] (-2.9971686085895866,-0.9971428442497587)-- (-1.9598157750067955,-0.9996485274226641);
\draw [line width=1pt] (-7.399653943384187,-1.0046598937684774)-- (-6.66,-1);
\draw [line width=1pt] (-7.06,-0.12)-- (-6.545215981423482,-0.5511312394726173);
\draw [line width=1pt] (-6.545215981423482,-0.5511312394726173)-- (-5.998977049730127,-1.0021542105955705);
\draw [line width=1pt] (-6.66,-1)-- (-5.998977049730127,-1.0021542105955705);
\draw (5.367932836151818,0.9646593959241845) node[anchor=north west] {$r_1 P_2$};
\draw (-7.404259153640635,1.3022503736279512) node[anchor=north west] {$r_2 P2$};
\draw [line width=1pt,dotted] (0.0030355131150047,2.075828704514527)-- (0.0030355131150047,1.3161226216218413);
\draw [line width=1pt] (0.0030355131150047,-0.27895748070662657)-- (0.0030355131150047,-0.9993162365969024);
\draw [line width=1pt] (-0.9600481890175837,-0.9971428442497587)-- (0.0030355131150047,-0.9993162365969024);
\draw [line width=1pt,dotted] (2.9989312241727784,-0.9996485274226641)-- (4.02,-1);
\draw [line width=1pt,dotted] (0.9993960521943549,-0.9996485274226641)-- (2.0417602521229568,-0.9996485274226641);
\draw [line width=1pt,dotted] (0.000008795653280872282,0.5564165387291554)-- (0.0030355131150047,-0.27895748070662657);
\draw (-3.437565165621393,-1.9798032627553214) node[anchor=north west] {$q_2$};
\draw (3.0610611551760885,-1.9798032627553214) node[anchor=north west] {$q_1$};
\draw (0.6697917297744177,1.9868907252639374) node[anchor=north west] {$h$};
\draw [line width=0.8pt] (-5.998722801439957,-1.993969849943333)-- (-0.9055039114889535,-2.0012627655088457);
\draw [line width=0.8pt] (1.034112644044173,-1.9864528004246156)-- (5.009814577362994,-1.9896298874580076);
\draw [line width=0.8pt] (0.40256221806254994,2.992924095416853)-- (0.40256221806254994,-0.4000261791755805);
\draw [line width=1pt] (-1.9598157750067955,-0.9996485274226641)-- (-0.9600481890175837,-0.9971428442497587);
\draw (1.485636592558517,-0.019980955711801807) node[anchor=north west] {$u_1$};
\draw (-3.0155764434916863,-0.0762461186624296) node[anchor=north west] {$v_1$};
\draw (-3.915819050701727,-0.0762461186624296) node[anchor=north west] {$v_2$};
\draw (2.8641330848488917,-0.1325112816130574) node[anchor=north west] {$u_2$};
\draw [line width=0.8pt] (2.9108594463914534,-1.4067713010116925)-- (5.009814577362994,-1.396353106865182);
\draw [line width=0.8pt] (-6.0007044977567086,-1.396353106865183)-- (-3.9572129739081756,-1.3995765833071183);
\draw [shift={(0.0008813077915958978,-0.9983956858362114)},line width=0.8pt,color=dtsfsf]  plot[domain=3.141174768114745:6.282767421704538,variable=\t]({1*2.998050178153337*cos(\t r)+0*2.998050178153337*sin(\t r)},{0*2.998050178153337*cos(\t r)+1*2.998050178153337*sin(\t r)});
\draw [shift={(-0.9788408128143735,-0.9996485274226641)},line width=0.8pt,color=rvwvcq]  plot[domain=0:3.141592653589793,variable=\t]({1*3.0206010649373303*cos(\t r)+0*3.0206010649373303*sin(\t r)},{0*3.0206010649373303*cos(\t r)+1*3.0206010649373303*sin(\t r)});
\draw (-5.153652635615533,-1.35) node[anchor=north west] {$s$};
\draw (3.9050385994355015,-1.35) node[anchor=north west] {$t$};
\draw [line width=1pt] (5.793077934508813,-10.338913238853765)-- (6.287416289098737,-10.586082416148727);
\draw [line width=1pt] (5.666236940529355,-9.65986456928823)-- (6.203900899707957,-9.350966220689065);
\draw [line width=1pt] (0.06693641282295731,-9.99028245728597)-- (1.0632969519023077,-9.99061474811173);
\draw [line width=1pt] (-5.935076150022173,-9.993120431284636)-- (-4.937814247205868,-9.993120431284636);
\draw [line width=1pt] (-6.463197273214205,-10.258057169168001)-- (-6.821740494734792,-10.454402266667373);
\draw [line width=1pt] (4.083900899707954,-9.990966220689065)-- (5.0648729790320814,-9.993120431284636);
\draw [line width=1pt] (5.0648729790320814,-9.993120431284636)-- (5.793077934508813,-10.338913238853765);
\draw [line width=1pt] (5.0648729790320814,-9.993120431284636)-- (5.666236940529355,-9.65986456928823);
\draw [line width=1pt,dotted] (-3.9355409780437496,-9.99061474811173)-- (-4.937814247205868,-9.993120431284636);
\draw (-0.14605313300968167,-10.007047379448233) node[anchor=north west] {$u$};
\draw [line width=1pt] (-5.935076150022173,-9.993120431284636)-- (-6.463197273214205,-10.258057169168001);
\draw [line width=1pt] (0.0639096953612342,-6.034362734812899)-- (0.06693641282295731,-6.915137516174538);
\draw [line width=1pt] (0.06693641282295731,-7.674843599067225)-- (0.0639096953612342,-8.434549681959911);
\draw [line width=1pt,dotted] (-2.9332677088816332,-9.988109064938824)-- (-1.8959148752988424,-9.99061474811173);
\draw [line width=1pt] (-6.8814976983215566,-9.643411646561276)-- (-6.4888075033228185,-9.78853628384342);
\draw [line width=1pt] (-6.4888075033228185,-9.78853628384342)-- (-5.935076150022173,-9.993120431284636);
\draw (-7.094800757412183,-8.375357653880027) node[anchor=north west] {$r_1 P_2$};
\draw (5.9024518841827796,-8.122164420602202) node[anchor=north west] {$r_2 P2$};
\draw [line width=1pt,dotted] (0.06693641282295731,-6.915137516174538)-- (0.06693641282295731,-7.674843599067225);
\draw [line width=1pt] (0.06693641282295731,-9.269923701395692)-- (0.06693641282295731,-9.99028245728597);
\draw [line width=1pt] (-0.8961472893096307,-9.988109064938824)-- (0.06693641282295731,-9.99028245728597);
\draw [line width=1pt,dotted] (3.0628321238807334,-9.99061474811173)-- (4.083900899707954,-9.990966220689065);
\draw [line width=1pt,dotted] (1.0632969519023077,-9.99061474811173)-- (2.105661151830911,-9.99061474811173);
\draw [line width=1pt,dotted] (0.0639096953612342,-8.434549681959911)-- (0.06693641282295731,-9.269923701395692);
\draw (-3.3531674211954514,-11.282229334855768) node[anchor=north west] {$q'_1 =q_2+t-s$};
\draw (3.117326318126716,-11.282229334855768) node[anchor=north west] {$q'_2=q_1-t+s$};
\draw (0.7260568927250453,-6.687402765361194) node[anchor=north west] {$h$};
\draw [line width=0.8pt] (-5.934821901732002,-10.984936070632399)-- (-0.8416030117810005,-10.99222898619791);
\draw [line width=0.8pt] (1.0980135437521261,-10.97741902111368)-- (5.073715477070952,-10.980596108147074);
\draw [line width=0.8pt] (0.4664631177705029,-5.998042125272213)-- (0.4664631177705029,-9.390992399864645);
\draw [line width=1pt] (-1.8959148752988424,-9.99061474811173)-- (-0.8961472893096307,-9.988109064938824);
\draw (1.5419017555091445,-9.022407027812248) node[anchor=north west] {$u_1$};
\draw (-2.959311280541059,-9.078672190762875) node[anchor=north west] {$v_1$};
\draw (2.8360005033735782,-8.994274446336933) node[anchor=north west] {$v_2$};
\draw (-4.08461453955361,-8.938009283386306) node[anchor=north west] {$u_2$};
\draw [line width=0.8pt] (2.9747603460994085,-10.39773752170076)-- (5.073715477070952,-10.387319327554247);
\draw [line width=0.8pt] (-5.936803598048755,-10.38731932755425)-- (-3.893312074200222,-10.390542803996185);
\draw (-5.069254891189591,-10.250782216497606) node[anchor=north west] {$t$};
\draw (3.961303762386129,-10.335179960923546) node[anchor=north west] {$s$};
\draw [line width=1pt] (5.835062372051567,-9.970936966878082)-- (6.5747163154357535,-9.966277073109605);
\draw [line width=1pt] (5.0648729790320814,-9.993120431284636)-- (5.835062372051567,-9.970936966878082);
\draw [line width=1pt,color=dtsfsf] (-3.9355409780437496,-9.99061474811173)-- (-2.9332677088816332,-9.988109064938824);
\draw [line width=1pt,color=rvwvcq] (2.105661151830911,-9.99061474811173)-- (3.0628321238807334,-9.99061474811173);
\draw (3.342386969929226,2.2024929808379956) node[anchor=north west] {$T$};
\draw (3.848773436484874,-6.65927018388588) node[anchor=north west] {$T'$};
\begin{scriptsize}
\draw [fill=ffffff] (-3.999441877751704,-0.999648527422664) circle (2.5pt);
\draw [fill=ffffff] (-5.001715146913821,-1.0021542105955705) circle (2.5pt);
\draw [fill=ffffff] (5.604841723994303,-1.217642963465425) circle (2.5pt);
\draw [fill=ffffff] (6.18,-1.44) circle (2.5pt);
\draw [fill=ffffff] (0.0030355131150047,-0.9993162365969024) circle (2.5pt);
\draw [fill=ffffff] (5.602336040821398,-0.6688983485991646) circle (2.5pt);
\draw [fill=ffffff] (6.14,-0.36) circle (2.5pt);
\draw [fill=ffffff] (0.9993960521943549,-0.9996485274226641) circle (2.5pt);
\draw [fill=ffffff] (4.02,-1) circle (2.5pt);
\draw [fill=ffffff] (-5.998977049730127,-1.0021542105955705) circle (2.5pt);
\draw [fill=ffffff] (5.000972079324126,-1.0021542105955694) circle (2.5pt);
\draw [fill=ffffff] (-6.48,-1.44) circle (2.5pt);
\draw [fill=ffffff] (-6.94,-1.88) circle (2.5pt);
\draw [fill=ffffff] (0.000008795653280872282,2.9566034858761667) circle (2.5pt);
\draw [fill=ffffff] (0.0030355131150047,2.075828704514527) circle (2.5pt);
\draw [fill=ffffff] (0.0030355131150047,1.3161226216218413) circle (2.5pt);
\draw [fill=ffffff] (0.000008795653280872282,0.5564165387291554) circle (2.5pt);
\draw [fill=ffffff] (0.0030355131150047,-0.27895748070662657) circle (2.5pt);
\draw [fill=ffffff] (-2.9971686085895866,-0.9971428442497587) circle (2.5pt);
\draw [fill=ffffff] (-1.9598157750067955,-0.9996485274226641) circle (2.5pt);
\draw [fill=ffffff] (-7.399653943384187,-1.0046598937684774) circle (2.5pt);
\draw [fill=ffffff] (-6.66,-1) circle (2.5pt);
\draw [fill=ffffff] (-7.06,-0.12) circle (2.5pt);
\draw [fill=ffffff] (-6.545215981423482,-0.5511312394726173) circle (2.5pt);
\draw [fill=ffffff] (2.9989312241727784,-0.9996485274226641) circle (2.5pt);
\draw [fill=ffffff] (2.0417602521229568,-0.9996485274226641) circle (2.5pt);
\draw [fill=ffffff] (-0.9600481890175837,-0.9971428442497587) circle (2.5pt);
\draw [fill=black,shift={(-5.998722801439957,-1.993969849943333)},rotate=270] (0,0) ++(0 pt,2.25pt) -- ++(1.9485571585149868pt,-3.375pt)--++(-3.8971143170299736pt,0 pt) -- ++(1.9485571585149868pt,3.375pt);
\draw [fill=black,shift={(-0.9055039114889535,-2.0012627655088457)},rotate=90] (0,0) ++(0 pt,2.25pt) -- ++(1.9485571585149868pt,-3.375pt)--++(-3.8971143170299736pt,0 pt) -- ++(1.9485571585149868pt,3.375pt);
\draw [fill=black,shift={(1.034112644044173,-1.9864528004246156)},rotate=270] (0,0) ++(0 pt,2.25pt) -- ++(1.9485571585149868pt,-3.375pt)--++(-3.8971143170299736pt,0 pt) -- ++(1.9485571585149868pt,3.375pt);
\draw [fill=black,shift={(5.009814577362994,-1.9896298874580076)},rotate=90] (0,0) ++(0 pt,2.25pt) -- ++(1.9485571585149868pt,-3.375pt)--++(-3.8971143170299736pt,0 pt) -- ++(1.9485571585149868pt,3.375pt);
\draw [fill=black,shift={(0.40256221806254994,2.992924095416853)},rotate=180] (0,0) ++(0 pt,2.25pt) -- ++(1.9485571585149868pt,-3.375pt)--++(-3.8971143170299736pt,0 pt) -- ++(1.9485571585149868pt,3.375pt);
\draw [fill=black,shift={(0.40256221806254994,-0.4000261791755805)}] (0,0) ++(0 pt,2.25pt) -- ++(1.9485571585149868pt,-3.375pt)--++(-3.8971143170299736pt,0 pt) -- ++(1.9485571585149868pt,3.375pt);
\draw [fill=black,shift={(2.9108594463914534,-1.4067713010116925)},rotate=270] (0,0) ++(0 pt,2.25pt) -- ++(1.9485571585149868pt,-3.375pt)--++(-3.8971143170299736pt,0 pt) -- ++(1.9485571585149868pt,3.375pt);
\draw [fill=black,shift={(5.009814577362994,-1.396353106865182)},rotate=90] (0,0) ++(0 pt,2.25pt) -- ++(1.9485571585149868pt,-3.375pt)--++(-3.8971143170299736pt,0 pt) -- ++(1.9485571585149868pt,3.375pt);
\draw [fill=black,shift={(-6.0007044977567086,-1.396353106865183)},rotate=270] (0,0) ++(0 pt,2.25pt) -- ++(1.9485571585149868pt,-3.375pt)--++(-3.8971143170299736pt,0 pt) -- ++(1.9485571585149868pt,3.375pt);
\draw [fill=black,shift={(-3.9572129739081756,-1.3995765833071183)},rotate=90] (0,0) ++(0 pt,2.25pt) -- ++(1.9485571585149868pt,-3.375pt)--++(-3.8971143170299736pt,0 pt) -- ++(1.9485571585149868pt,3.375pt);
\draw [fill=ffffff] (-3.9355409780437496,-9.99061474811173) circle (2.5pt);
\draw [fill=ffffff] (-4.937814247205868,-9.993120431284636) circle (2.5pt);
\draw [fill=ffffff] (5.793077934508813,-10.338913238853765) circle (2.5pt);
\draw [fill=ffffff] (6.287416289098737,-10.586082416148727) circle (2.5pt);
\draw [fill=ffffff] (0.06693641282295731,-9.99028245728597) circle (2.5pt);
\draw [fill=ffffff] (5.666236940529355,-9.65986456928823) circle (2.5pt);
\draw [fill=ffffff] (6.203900899707957,-9.350966220689065) circle (2.5pt);
\draw [fill=ffffff] (1.0632969519023077,-9.99061474811173) circle (2.5pt);
\draw [fill=ffffff] (4.083900899707954,-9.990966220689065) circle (2.5pt);
\draw [fill=ffffff] (-5.935076150022173,-9.993120431284636) circle (2.5pt);
\draw [fill=ffffff] (5.0648729790320814,-9.993120431284636) circle (2.5pt);
\draw [fill=ffffff] (-6.463197273214205,-10.258057169168001) circle (2.5pt);
\draw [fill=ffffff] (-6.821740494734792,-10.454402266667373) circle (2.5pt);
\draw [fill=ffffff] (0.0639096953612342,-6.034362734812899) circle (2.5pt);
\draw [fill=ffffff] (0.06693641282295731,-6.915137516174538) circle (2.5pt);
\draw [fill=ffffff] (0.06693641282295731,-7.674843599067225) circle (2.5pt);
\draw [fill=ffffff] (0.0639096953612342,-8.434549681959911) circle (2.5pt);
\draw [fill=ffffff] (0.06693641282295731,-9.269923701395692) circle (2.5pt);
\draw [fill=ffffff] (-2.9332677088816332,-9.988109064938824) circle (2.5pt);
\draw [fill=ffffff] (-1.8959148752988424,-9.99061474811173) circle (2.5pt);
\draw [fill=ffffff] (-6.8814976983215566,-9.643411646561276) circle (2.5pt);
\draw [fill=ffffff] (-6.4888075033228185,-9.78853628384342) circle (2.5pt);
\draw [fill=ffffff] (3.0628321238807334,-9.99061474811173) circle (2.5pt);
\draw [fill=ffffff] (2.105661151830911,-9.99061474811173) circle (2.5pt);
\draw [fill=ffffff] (-0.8961472893096307,-9.988109064938824) circle (2.5pt);
\draw [fill=black,shift={(-5.934821901732002,-10.984936070632399)},rotate=270] (0,0) ++(0 pt,2.25pt) -- ++(1.9485571585149868pt,-3.375pt)--++(-3.8971143170299736pt,0 pt) -- ++(1.9485571585149868pt,3.375pt);
\draw [fill=black,shift={(-0.8416030117810005,-10.99222898619791)},rotate=90] (0,0) ++(0 pt,2.25pt) -- ++(1.9485571585149868pt,-3.375pt)--++(-3.8971143170299736pt,0 pt) -- ++(1.9485571585149868pt,3.375pt);
\draw [fill=black,shift={(1.0980135437521261,-10.97741902111368)},rotate=270] (0,0) ++(0 pt,2.25pt) -- ++(1.9485571585149868pt,-3.375pt)--++(-3.8971143170299736pt,0 pt) -- ++(1.9485571585149868pt,3.375pt);
\draw [fill=black,shift={(5.073715477070952,-10.980596108147074)},rotate=90] (0,0) ++(0 pt,2.25pt) -- ++(1.9485571585149868pt,-3.375pt)--++(-3.8971143170299736pt,0 pt) -- ++(1.9485571585149868pt,3.375pt);
\draw [fill=black,shift={(0.4664631177705029,-5.998042125272213)},rotate=180] (0,0) ++(0 pt,2.25pt) -- ++(1.9485571585149868pt,-3.375pt)--++(-3.8971143170299736pt,0 pt) -- ++(1.9485571585149868pt,3.375pt);
\draw [fill=black,shift={(0.4664631177705029,-9.390992399864645)}] (0,0) ++(0 pt,2.25pt) -- ++(1.9485571585149868pt,-3.375pt)--++(-3.8971143170299736pt,0 pt) -- ++(1.9485571585149868pt,3.375pt);
\draw [fill=black,shift={(2.9747603460994085,-10.39773752170076)},rotate=270] (0,0) ++(0 pt,2.25pt) -- ++(1.9485571585149868pt,-3.375pt)--++(-3.8971143170299736pt,0 pt) -- ++(1.9485571585149868pt,3.375pt);
\draw [fill=black,shift={(5.073715477070952,-10.387319327554247)},rotate=90] (0,0) ++(0 pt,2.25pt) -- ++(1.9485571585149868pt,-3.375pt)--++(-3.8971143170299736pt,0 pt) -- ++(1.9485571585149868pt,3.375pt);
\draw [fill=black,shift={(-5.936803598048755,-10.38731932755425)},rotate=270] (0,0) ++(0 pt,2.25pt) -- ++(1.9485571585149868pt,-3.375pt)--++(-3.8971143170299736pt,0 pt) -- ++(1.9485571585149868pt,3.375pt);
\draw [fill=black,shift={(-3.893312074200222,-10.390542803996185)},rotate=90] (0,0) ++(0 pt,2.25pt) -- ++(1.9485571585149868pt,-3.375pt)--++(-3.8971143170299736pt,0 pt) -- ++(1.9485571585149868pt,3.375pt);
\draw [fill=ffffff] (5.835062372051567,-9.970936966878082) circle (2.5pt);
\draw [fill=ffffff] (6.5747163154357535,-9.966277073109605) circle (2.5pt);
\end{scriptsize}
\end{tikzpicture} 
  \caption{The 2-switch of type II.}\label{TypeIISwitch}
\end{figure}
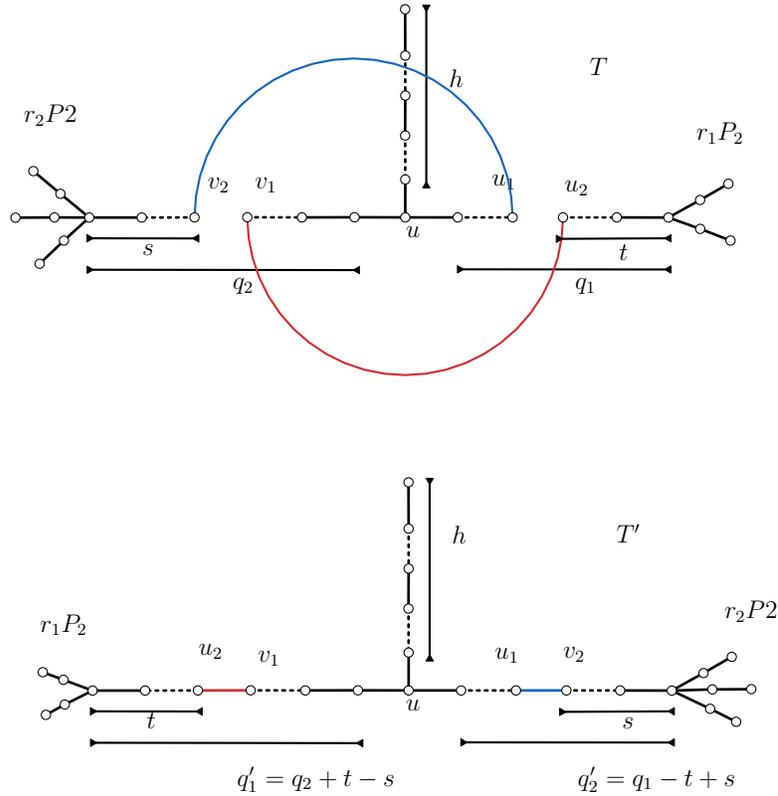

We can always assume that $\lambda= \rho(T) >\sqrt{r_2+2}$ for all $T$ in $\mathfrak{F}(n)$ because the sun $S_{r_2+1}$, whose spectral radius is $\sqrt{r_2+2}$, is a proper subgraph of $T$, since $q_2 \geq 2$. But in fact, we need a larger lower bound for $\lambda =\rho(T)$ for our results.

Let $T_j:=[\underbrace{2,2,...,2}_{r_2 \text{ times}},j]$ be the starlike tree composed by $r_2$ legs of $P_2$ and a path of length $j$, with $j\geq 3$ as Figure \ref{fig:T_j} illustrates.
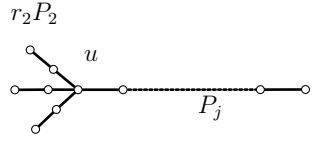
\begin{figure}[H]
  \centering
\definecolor{ffffff}{rgb}{1,1,1}
\definecolor{cqcqcq}{rgb}{0.7529411764705882,0.7529411764705882,0.7529411764705882}
\begin{tikzpicture}[line cap=round,line join=round,>=triangle 45,x=1cm,y=1cm,scale=0.6, every node/.style={scale=0.8}]
\clip(-8.86,-2.2) rectangle (8.86,1.5);
\draw [line width=1pt] (-1.2389770497301233,-0.90215421059557)-- (-0.24171514691381868,-0.90215421059557);
\draw [line width=1pt] (-1.72,-1.34)-- (-2.18,-1.78);
\draw [line width=1pt] (-1.2389770497301233,-0.90215421059557)-- (-1.72,-1.34);
\draw [line width=1pt] (-2.639653943384183,-0.9046598937684769)-- (-1.9,-0.9);
\draw [line width=1pt] (-2.3,-0.02)-- (-1.7852159814234794,-0.4511312394726168);
\draw [line width=1pt] (-1.7852159814234794,-0.4511312394726168)-- (-1.2389770497301233,-0.90215421059557);
\draw [line width=1pt] (-1.9,-0.9)-- (-1.2389770497301233,-0.90215421059557);
\draw (-2.92,1.22) node[anchor=north west] {$r_2 P_2$};
\draw [line width=1pt] (2.800184224993207,-0.8996485274226635)-- (3.799951810982418,-0.8971428442497582);
\draw [line width=1pt,dash pattern=on 1pt off 1pt] (-0.24171514691381868,-0.90215421059557)-- (2.800184224993207,-0.8996485274226635);
\draw (1.22,-0.84) node[anchor=north west] {$P_j$};
\draw (-1.3,0.18) node[anchor=north west] {$u$};
\begin{scriptsize}
\draw [fill=ffffff] (-0.24171514691381868,-0.90215421059557) circle (2.5pt);
\draw [fill=ffffff] (-1.2389770497301233,-0.90215421059557) circle (2.5pt);
\draw [fill=ffffff] (-1.72,-1.34) circle (2.5pt);
\draw [fill=ffffff] (-2.18,-1.78) circle (2.5pt);
\draw [fill=ffffff] (2.800184224993207,-0.8996485274226635) circle (2.5pt);
\draw [fill=ffffff] (-2.639653943384183,-0.9046598937684769) circle (2.5pt);
\draw [fill=ffffff] (-1.9,-0.9) circle (2.5pt);
\draw [fill=ffffff] (-2.3,-0.02) circle (2.5pt);
\draw [fill=ffffff] (-1.7852159814234794,-0.4511312394726168) circle (2.5pt);
\draw [fill=ffffff] (3.799951810982418,-0.8971428442497582) circle (2.5pt);
\end{scriptsize}
\end{tikzpicture} 
  \caption{The graph $T_j$. }\label{fig:T_j}
\end{figure}

\begin{theorem}\label{th:T_j} Let $T \in \mathfrak{F}(n)$. If $j \geq 3$, then $$\rho(T)> \rho(T_j).$$
\end{theorem}
This result may be proven by our comparison method that will be explained next, but there is a simpler proof, for which we need the following definitions and known result of Lemma \ref{lem:hoff}.

An internal path in a graph $G$, denoted by $v_1 v_2 ,\ldots v_{r-1} v_r$, is
a path beginning at $v_1$ and ending at $v_r$, where $v_1$ and $v_r$ both
have degree bigger than two, while all other vertices have degree two. The
vertices $v_1$ and $v_r$ are not necessarily distinct. We denote by $C_n$, the cycle on $n$ vertices and by $W_n$ the tree with $n$ vertices where two vertices have degree three and the distance between them is $n-5$.  The following result appears in the work by Hoffman and Smith \cite{Hoff2}.

\begin{lemma}\label{lem:hoff}
Let $G$ be a graph with $n$ vertices, $G \neq C_n, W_n$. Let $G'$ be the
graph with $n+1$ vertices obtained from $G$ by inserting a new vertex of
degree two in an edge $e$. Then
\begin{itemize}
\item[(a)] if $e$ lies on an internal path then $\lambda(G')< \lambda(G)$;
\item[(b)] if $e$ does not lie on an internal path then     $\lambda(G')>\lambda(G)$.
\end{itemize}
\end{lemma}

\begin{proof} (Theorem~\ref{th:T_j}) Let $T =[h,q_1,q_2]~\in \mathfrak{F}(n)$ and $3 \leq j$.

Let $H'_1$ be the tree obtained by adding an edge in the internal path starting from $u$ to the vertex $r_1P_2$. We notice that $T \not =W_n$, as it has a vertex $u$ of degre $r_2+1 > 3$, and hence, by Lemma \ref{lem:hoff}, $\rho(H'_1) < \rho(T)$ . We now remove the pendant vertex of the path $P_h$ in $T$, obtaining a tree $H_1$, a proper subtree of $H'_1$. If follows that $\rho(H_1) < \rho(T)$. We apply this process successively $h$ times, obtaining a tree $H_h$ composed by the starlikes $r_2P_2$ and $r_1P_2$ linked in their centers by a path of length $q_2+h+q_1+1$ such that $\rho(H_h) < \rho(T).$

Now, if $3\leq j \leq q_2+h+q_1+1$, we see that $T_j$ is a proper subtree of $H_h$ and therefore, $\rho(T_j) < \rho(H_h) < \rho(T)$.

For $j > q_2+h+q_1+1$, we keep adding edges in the internal path starting at $r_2P_2$ and ending at $r_1P_2$ until the length of the path is at least $j$, obtaining  a tree $H_j$. This operation, according the Lemma \ref{lem:hoff}, decreases the spectral radius. As $T_j$ is a proper subtree of $H_j$, it follows that $\rho(T_j) < \rho(H_j) < \rho(T)$.\\
\end{proof}

\subsection{Our tool}

We would like to recall the algorithm {\bf Diagonalize($T$, $\alpha$)}.  For a tree~$T$ and a real number~$\alpha$ this algorithm outputs a sequence $(d_v)_{v\in V(T)}$.\\

\vspace*{.5cm}

{\bf Algorithm Diagonalize($T$, $\alpha$)}
\begin{itemize}
\item[{\bf 1.}] List the vertices of~$T$ in postorder as $v_1,\dots,v_n$.

\item[{\bf 2.}] For each $i=1,\dots,n$
                 set $d_{v_i}\leftarrow\alpha$.

\item[{\bf 3.}] For each $i=1,\dots,n$:

\item[{\bf 4.}] \hspace*{18pt} If $v_i$ has a child $v_j$ such that $d_{v_j}=0$,\\
\hspace*{18pt} then \\
\hspace*{25pt} set $d_{v_i}\leftarrow -\frac12$ and $d_{v_j}\leftarrow 2$. \\ \hspace*{25pt} Further, if $v_i$ has a parent $v_p$, remove the edge $v_pv_i$ from $T$.

\item[{\bf 5.}] \hspace*{18pt} Otherwise,
set $d_{v_i}\leftarrow d_{v_i} - \sum d_{v_j}^{-1}$,
summing over all children $v_j$ of $v_i$.
\end{itemize}

The above algorithm of Jacobs and Trevisan~\cite{JaTr} can be used to
estimate eigenvalues of a given tree. It is based on diagonalization of the
matrix $A(T)+ \alpha I$, where $A(T)$ is the adjacency matrix of~$T$ and
$\alpha$ is a real number. Its nice feature is that it can be easily executed
manually directly on the drawing of a tree. The authors proved that this
algorithm diagonalizes $A(T)+\alpha I$ and, additionally, the following result holds.

\begin{theorem} \label{th-diagonalization}
For a tree $T$, let $(d_v)_{v\in V(T)}$ be the values produced by {\bf Diagonalize}$ (T, -\alpha)$. Then the diagonal matrix $D=\mathrm{diag}(d_v)_{v\in V(T)}$ is congruent to $A(T)+\alpha I$, hence
the number of $($ positive $|$ negative $|$ zero $)$ entries in
$(d_v)_{v\in V(T)}$ is equal to the number of eigenvalues of $A(T)$ that are
$($ greater than $\alpha$ $|$ smaller than $\alpha$ $|$ equal to $\alpha$
$)$.
\end{theorem}
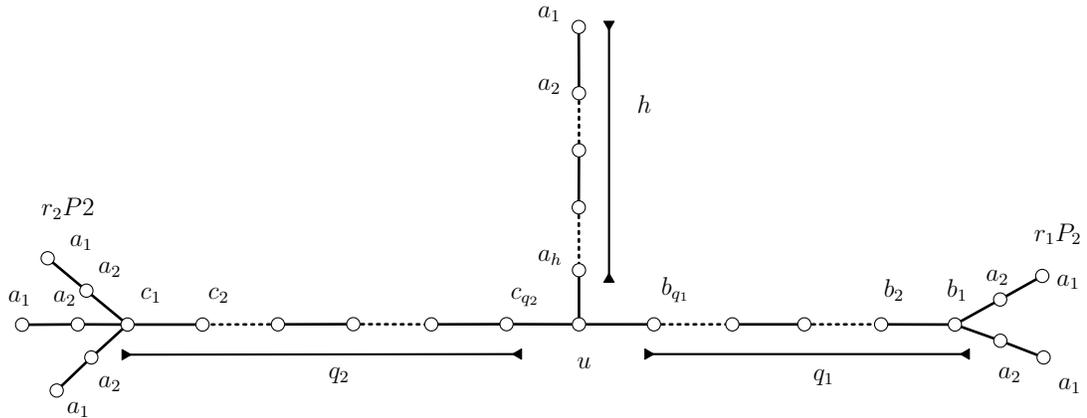
\begin{figure}[H]
  \centering
\definecolor{ffffff}{rgb}{1,1,1}
\begin{tikzpicture}[line cap=round,line join=round,>=triangle 45,x=1cm,y=1cm,scale=1, every node/.style={scale=0.8}]
\clip(-9,-3.0) rectangle (9,3.3);
\draw [line width=1pt] (5.604841723994303,-1.217642963465425)-- (6.18,-1.44);
\draw [line width=1pt] (5.602336040821398,-0.6688983485991646)-- (6.16206491229781,-0.35573853705499126);
\draw [line width=1pt] (0.0030355131150047,-0.9993162365969024)-- (0.9993960521943549,-0.9996485274226641);
\draw [line width=1pt] (-5.998977049730127,-1.0021542105955705)-- (-5.001715146913821,-1.0021542105955705);
\draw [line width=1pt] (-6.48,-1.44)-- (-6.94,-1.88);
\draw [line width=1pt] (4.02,-1)-- (5.000972079324126,-1.0021542105955694);
\draw [line width=1pt] (5.000972079324126,-1.0021542105955694)-- (5.604841723994303,-1.217642963465425);
\draw [line width=1pt] (5.000972079324126,-1.0021542105955694)-- (5.602336040821398,-0.6688983485991646);
\draw [line width=1pt,dotted] (-3.999441877751704,-0.999648527422664)-- (-5.001715146913821,-1.0021542105955705);
\draw (-0.14,-1.32) node[anchor=north west] {$u$};
\draw [line width=1pt] (-5.998977049730127,-1.0021542105955705)-- (-6.48,-1.44);
\draw [line width=1pt] (0.000008795653280872282,2.9566034858761667)-- (0.0030355131150047,2.075828704514527);
\draw [line width=1pt] (0.0030355131150047,1.3161226216218413)-- (0.000008795653280872282,0.5564165387291554);
\draw [line width=1pt,dotted] (-2.9971686085895866,-0.9971428442497587)-- (-1.9598157750067955,-0.9996485274226641);
\draw [line width=1pt] (-7.399653943384187,-1.0046598937684774)-- (-6.66,-1);
\draw [line width=1pt] (-7.06,-0.12)-- (-6.545215981423482,-0.5511312394726173);
\draw [line width=1pt] (-6.545215981423482,-0.5511312394726173)-- (-5.998977049730127,-1.0021542105955705);
\draw [line width=1pt] (-6.66,-1)-- (-5.998977049730127,-1.0021542105955705);
\draw (5.94,0.46) node[anchor=north west] {$r_1 P_2$};
\draw (-7.26,0.8) node[anchor=north west] {$r_2 P2$};
\draw [line width=1pt,dotted] (0.0030355131150047,2.075828704514527)-- (0.0030355131150047,1.3161226216218413);
\draw [line width=1pt] (0.0030355131150047,-0.27895748070662657)-- (0.0030355131150047,-0.9993162365969024);
\draw [line width=1pt] (2.9989312241727784,-0.9996485274226641)-- (2.0417602521229568,-0.9996485274226641);
\draw [line width=1pt] (-0.9600481890175837,-0.9971428442497587)-- (0.0030355131150047,-0.9993162365969024);
\draw [line width=1pt] (-3.999441877751704,-0.999648527422664)-- (-2.9971686085895866,-0.9971428442497587);
\draw [line width=1pt,dotted] (2.9989312241727784,-0.9996485274226641)-- (4.02,-1);
\draw [line width=1pt,dotted] (0.9993960521943549,-0.9996485274226641)-- (2.0417602521229568,-0.9996485274226641);
\draw [line width=1pt,dotted] (0.000008795653280872282,0.5564165387291554)-- (0.0030355131150047,-0.27895748070662657);
\draw (-3.44,-1.46) node[anchor=north west] {$q_2$};
\draw (3,-1.48) node[anchor=north west] {$q_1$};
\draw (0.66,2.16) node[anchor=north west] {$h$};
\draw [line width=0.8pt] (-6.031550930977896,-1.4030635182604188)-- (-0.799684465951645,-1.3955464687417012);
\draw [line width=0.8pt] (0.9192141906613855,-1.3955464687417012)-- (5.153818752871347,-1.3955464687417012);
\draw [line width=0.8pt] (0.40256221806254994,2.992924095416853)-- (0.40256221806254994,-0.4000261791755805);
\draw [line width=1pt] (-1.9598157750067955,-0.9996485274226641)-- (-0.9600481890175837,-0.9971428442497587);
\draw (-6.88,0.30) node[anchor=north west] {$a_1$};
\draw (-7.7,-0.44) node[anchor=north west] {$a_1$};
\draw (-6.92,-1.92) node[anchor=north west] {$a_1$};
\draw (-6.5,-0.06) node[anchor=north west] {$a_2$};
\draw (-7.1,-0.44) node[anchor=north west] {$a_2$};
\draw (-6.5,-1.58) node[anchor=north west] {$a_2$};
\draw (5.3,-0.16) node[anchor=north west] {$a_2$};
\draw (5.46,-1.46) node[anchor=north west] {$a_2$};
\draw (6.24,-0.22) node[anchor=north west] {$a_1$};
\draw (6.26,-1.6) node[anchor=north west] {$a_1$};
\draw (-0.66,3.34) node[anchor=north west] {$a_1$};
\draw (-0.66,2.38) node[anchor=north west] {$a_2$};
\draw (-0.66,0.1) node[anchor=north west] {$a_{h}$};
\draw (4.78,-0.28) node[anchor=north west] {$b_1$};
\draw (3.94,-0.28) node[anchor=north west] {$b_2$};
\draw (0.98,-0.26) node[anchor=north west] {$b_{q_1}$};
\draw (-1.02,-0.4) node[anchor=north west] {$c_{q_2}$};
\draw (-5.04,-0.4) node[anchor=north west] {$c_2$};
\draw (-5.94,-0.4) node[anchor=north west] {$c_1$};
\begin{scriptsize}
\draw [fill=ffffff] (-3.999441877751704,-0.999648527422664) circle (2.5pt);
\draw [fill=ffffff] (-5.001715146913821,-1.0021542105955705) circle (2.5pt);
\draw [fill=ffffff] (5.604841723994303,-1.217642963465425) circle (2.5pt);
\draw [fill=ffffff] (6.18,-1.44) circle (2.5pt);
\draw [fill=ffffff] (0.0030355131150047,-0.9993162365969024) circle (2.5pt);
\draw [fill=ffffff] (5.602336040821398,-0.6688983485991646) circle (2.5pt);
\draw [fill=ffffff] (6.16206491229781,-0.35573853705499126) circle (2.5pt);
\draw [fill=ffffff] (0.9993960521943549,-0.9996485274226641) circle (2.5pt);
\draw [fill=ffffff] (4.02,-1) circle (2.5pt);
\draw [fill=ffffff] (-5.998977049730127,-1.0021542105955705) circle (2.5pt);
\draw [fill=ffffff] (5.000972079324126,-1.0021542105955694) circle (2.5pt);
\draw [fill=ffffff] (-6.48,-1.44) circle (2.5pt);
\draw [fill=ffffff] (-6.94,-1.88) circle (2.5pt);
\draw [fill=ffffff] (0.000008795653280872282,2.9566034858761667) circle (2.5pt);
\draw [fill=ffffff] (0.0030355131150047,2.075828704514527) circle (2.5pt);
\draw [fill=ffffff] (0.0030355131150047,1.3161226216218413) circle (2.5pt);
\draw [fill=ffffff] (0.000008795653280872282,0.5564165387291554) circle (2.5pt);
\draw [fill=ffffff] (0.0030355131150047,-0.27895748070662657) circle (2.5pt);
\draw [fill=ffffff] (-2.9971686085895866,-0.9971428442497587) circle (2.5pt);
\draw [fill=ffffff] (-1.9598157750067955,-0.9996485274226641) circle (2.5pt);
\draw [fill=ffffff] (-7.399653943384187,-1.0046598937684774) circle (2.5pt);
\draw [fill=ffffff] (-6.66,-1) circle (2.5pt);
\draw [fill=ffffff] (-7.06,-0.12) circle (2.5pt);
\draw [fill=ffffff] (-6.545215981423482,-0.5511312394726173) circle (2.5pt);
\draw [fill=ffffff] (2.9989312241727784,-0.9996485274226641) circle (2.5pt);
\draw [fill=ffffff] (2.0417602521229568,-0.9996485274226641) circle (2.5pt);
\draw [fill=ffffff] (-0.9600481890175837,-0.9971428442497587) circle (2.5pt);
\draw [fill=black,shift={(-6.031550930977896,-1.4030635182604188)},rotate=270] (0,0) ++(0 pt,2.25pt) -- ++(1.9485571585149868pt,-3.375pt)--++(-3.8971143170299736pt,0 pt) -- ++(1.9485571585149868pt,3.375pt);
\draw [fill=black,shift={(-0.799684465951645,-1.3955464687417012)},rotate=90] (0,0) ++(0 pt,2.25pt) -- ++(1.9485571585149868pt,-3.375pt)--++(-3.8971143170299736pt,0 pt) -- ++(1.9485571585149868pt,3.375pt);
\draw [fill=black,shift={(0.9192141906613855,-1.3955464687417012)},rotate=270] (0,0) ++(0 pt,2.25pt) -- ++(1.9485571585149868pt,-3.375pt)--++(-3.8971143170299736pt,0 pt) -- ++(1.9485571585149868pt,3.375pt);
\draw [fill=black,shift={(5.153818752871347,-1.3955464687417012)},rotate=90] (0,0) ++(0 pt,2.25pt) -- ++(1.9485571585149868pt,-3.375pt)--++(-3.8971143170299736pt,0 pt) -- ++(1.9485571585149868pt,3.375pt);
\draw [fill=black,shift={(0.40256221806254994,2.992924095416853)},rotate=180] (0,0) ++(0 pt,2.25pt) -- ++(1.9485571585149868pt,-3.375pt)--++(-3.8971143170299736pt,0 pt) -- ++(1.9485571585149868pt,3.375pt);
\draw [fill=black,shift={(0.40256221806254994,-0.4000261791755805)}] (0,0) ++(0 pt,2.25pt) -- ++(1.9485571585149868pt,-3.375pt)--++(-3.8971143170299736pt,0 pt) -- ++(1.9485571585149868pt,3.375pt);
\end{scriptsize}
\end{tikzpicture} 
  \caption{The algorithm Diagonalize($T$, $-\lambda$) applied to a generic member of trees $\mathfrak{F}(n)$.}\label{Main_Family_Diago}
\end{figure}
Let $(d_v)_{v\in V}$ be the sequence obtained by executing {\bf Diagonalize} ($T,-\lambda$), when $T \in  \mathfrak{F}(n)$ (see Figure~\ref{Main_Family_Diago}).
Since we are going to use this algorithm in different trees, it  is useful to adopt a new notation, recording the tree we are using and the vertex where we are applying, which is
$$f_{T}(v):=d_v, \; \forall v \in V(T).$$
Taking $\lambda=\rho(T)$  we will reason that $f_{T}(v_n)=0$ and $f_{T}(v_i) <0$ for $i \not = n$.

For notation simplicity, let us rename the values in the extremity of each leaf $v_{1}$  by $a_1=-\lambda$, then $f_{T}(v_{j+1})= a_{j+1} = -\lambda-\frac{1}{a_{j}}$ for $j \geq 1$. In particular, $a_1 \in (-\infty, \, -2)$ and $a_{2} = -\lambda-\frac{1}{a_{1}}=-\lambda+\frac{1}{\lambda} \in (-\infty, \, -\frac{3}{2})$.\\

On the vertex with $r_1\, P_2$ and $r_2\, P_2$ the algorithm produces, respectively
$$b_1 =-\lambda-\frac{r_1}{a_{2}}, \text{ and } c_1 =-\lambda-\frac{r_2}{a_{2}}.$$
From these vertices towards the root $u$ we obtain two sequences $b_{j}$ and $c_{j}$, obeying the same relation as $a_{j}$, that is, $b_{j+1} = -\lambda-\frac{1}{b_{j}}$  and $c_{j+1} = -\lambda-\frac{1}{c_{j}}$.

The vertices labeled with the values of the numeric sequences generated by the application of the algorithm {\bf Diagonalize}
$(T,-\lambda)$ appear in Figure \ref{Main_Family_Diago}.

By the application of the algorithm, we know that $a_j, b_j, c_j<0$ for all indices appearing in the picture and
\begin{equation}\label{eq-root_main}
f_{T}(u)= -\lambda -\frac{1}{a_{h}}-\frac{1}{b_{q_1}}-\frac{1}{c_{q_2}} =0,
\end{equation}
otherwise, if some previous vertex produces zero then, the step (4) of  the algorithm will produce a positive value which is not possible because  $\lambda$ is the index of $T$.
Now suppose that we have a new tree $\tilde{T}$ with new parameters $[h', q'_1, q'_2]$ (same number of vertices and with the same $r_1$ and $r_2$). We now execute {\bf Diagonalize}$(\tilde{T},-\lambda)$.\\

Since the tree $\tilde{T}$ has the same properties of $T$ we obtain the same sequences and the same formula at the root $u$. More precisely, the execution of {\bf Diagonalize}$(\tilde{T},-\lambda)$ produces the same sequences, $a_j, b_j, c_j$  and
\begin{equation}\label{eq-root_main_after}
f_{\tilde{T}}(u)= -\lambda -\frac{1}{a_{h}}-\frac{1}{b_{q'_1}}-\frac{1}{c_{q'_2}}
\end{equation}
From Theorem~\ref{th-diagonalization} it follows that $\rho(T) > \rho(\tilde{T})$ if and only if  $-\lambda -\frac{1}{a_{h}}-\frac{1}{b_{q'_1}}-\frac{1}{c_{q'_2}} <0$ and all $a_j, b_j, c_j<0$.

We will see that in order to determine the sign of Equation (\ref{eq-root_main_after}), we need to deal in great detail with the recurrences appearing when the algorithm {\bf Diagonalize}$(\tilde{T},-\lambda)$ is implemented. For that we will determine some analytical properties of the recurrences $a_j, b_j$ and $c_j$.

\section{Analytical properties of recurrence sequences}
\label{sec:analytical properties}
As we observed in the previous section, the only information needed is the sign of the numeric sequences generated.  All the recurrence relations are of the same kind, differing only by the initial value. More precisely, they are of the form
\begin{equation}\label{phi-def}
 z_{j+1} = \varphi(z_j) \mbox{~~~where ~~ }\varphi(t) = -\lambda - \frac{1}{t}
\end{equation}
for $t\neq 0$ and $\lambda=\rho(T)>\rho(T_j)>2$.
Hence it depends only on the analytical behavior of the function $\varphi(t)$.

In \cite{BeOlTr} this sequence was extensively studied and its behaviour can be summarized by the following result whose proof is a combination of the results found in \cite{BeOlTr}.
\begin{theorem}\label{thm:results_on_the_sequence}
   Let $z_j$ be the recurrence  in formula~\eqref{phi-def}, then
   \begin{itemize}
     \item[(a)] $\varphi(t)=t$ has two fixed points $\theta=\frac{-\lambda-\sqrt{\lambda^2-4}}{2}<-1$ and $\theta^{-1}$;
     \item[(b)] $\theta(\lambda)$ is decreasing as a function of $\lambda$;
     \item[(c)] $z_{j}= \theta + \frac{\theta^{-1}-\theta}{\beta(\theta^2)^j+1},$ where the constant $\beta \in \mathbb{R}$ is obtained by choosing the value $z_{1}$;
     \item[(d)] The sequence $a_j$  obtained from $z_{j}$ by taking $a_1 =-\lambda$ is given by  $a_j=\theta-\frac{\sqrt{\lambda^2-4}}{(\theta^2)^j-1} <0,$ for any $j\geq 1$. In particular, $a_j$ is increasing and $\displaystyle \lim_{j \to \infty}a_{j}=\theta$.
\end{itemize}
\end{theorem}

We also need to understand the sequences  $b_j$ and $c_j$ obtained from $z_{j}$ by considering $b_{1}= -\lambda - \frac{r_1}{a_2}$ and $c_{1}= -\lambda - \frac{r_2}{a_2}$, $2 \leq r_1 < r_2 $, respectively.

We are going to study both at the same time by considering a parametric sequence
\begin{equation}\label{b_j_c_j_def}
 z_{j+1}(r) = \varphi(z_j(r))
\end{equation}
with $z_1(r):=-\lambda - \frac{r}{a_2}$ for some $r\geq 2$.

The main facts can be summarized in the following result.
\begin{theorem}\label{thm:results_on_the_sequence zjr}
   Let $z_j$ be the recurrence  in formula~\eqref{b_j_c_j_def}, then
   \begin{itemize}
     \item[(a)] $z_{j}(r)= \theta + \frac{\theta^{-1}-\theta}{\beta(\theta^2)^j+1},$ where the constant $\beta:= \beta(r) \in \mathbb{R}$ is given by $$\beta:= \frac{ r   - a_2 \theta}{a_2\theta - r \theta^2};$$
     \item[(b)] $\beta(r)$ is a continuous function of $r$ in $(2,\infty)\setminus r_*$, where $r_*:=\frac{a_2}{\theta}$. Moreover $\beta(r)$ has a single root in $r^*=a_2 \theta$ and $\beta(r)>0$, for $r \in (r_*, r^*)$

     \item[(c)] The sequence $z_j(r)<0$, for $j\geq 1$, is decreasing and $\displaystyle \lim_{j \to \infty}z_{j}=\theta$.
   \end{itemize}
\end{theorem}
\begin{proof}
{\bf (a)} Using Theorem~\ref{thm:results_on_the_sequence} (c) and the fact that $-\lambda= \theta+\theta^{-1}$  and $\theta-\theta^{-1}=
-\sqrt{\lambda^2-4}$ we obtain
$$z_1=-\lambda - \frac{r}{a_2}= \theta+\theta^{-1}- \frac{r}{a_2}
\text{ and }
z_{1}= \theta + \frac{\theta^{-1}-\theta}{\beta(\theta^2)^1 +1},$$
producing
$$\beta:= \beta(r)= \frac{\theta^{-1} \frac{r}{a_2} -1}{1- \theta \frac{r}{a_2}}=\frac{ r \theta^{-1}   - a_2}{a_2 - r \theta}=\frac{ r   - a_2 \theta}{a_2\theta - r \theta^2}.$$
{\bf (b)} We notice that $\beta(r)$ is a rational function, hence continuous, except for the roots of the denominator. Thus, the discontinuity  occurs at $r_*:=\frac{a_2}{\theta}$. Also, $\beta(r)$ has only one possible root in $r^*=a_2 \theta$.

Additionally, $\displaystyle\lim_{r\to r_*^{+}}\beta(r)=+\infty$. To see that we just take $r:=r_* +\delta$ for $\delta>0$, then
$$\beta(r_* +\delta)= \frac{ r_* +\delta   - a_2 \theta}{a_2\theta - (r_* +\delta) \theta^2}=\frac{ \delta +( r_* - a_2 \theta)}{(a_2\theta -r_* \theta^2) - \delta \theta^2}=\frac{ \delta +( r_* - a_2 \theta)}{- \delta \theta^2} \stackrel{\delta \to 0^+}{\to} +\infty$$
because $r_* - a_2 \theta= \frac{a_2}{\theta}- a_2 \theta=a_2(\frac{1-\theta^2}{\theta})<0$.
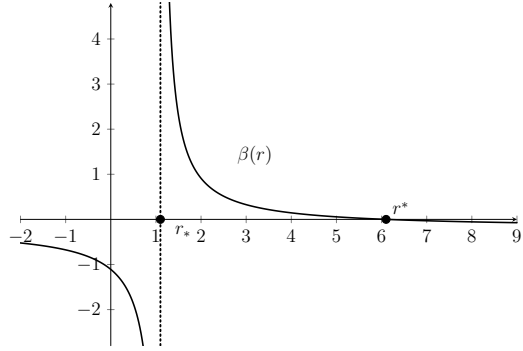
\begin{figure}[H]
  \centering
\definecolor{dtsfsf}{rgb}{0.8274509803921568,0.1843137254901961,0.1843137254901961}
\begin{tikzpicture}[line cap=round,line join=round,>=triangle 45,x=1cm,y=1cm,scale=0.6, every node/.style={scale=1}]
\begin{axis}[
x=1cm,y=1cm,
axis lines=middle,
xmin=-2,
xmax=9,
ymin=-2.800000000000001,
ymax=4.800000000000001,
xtick={-9,-8,...,9},
ytick={-4,-3,...,4},]
\clip(-9,-4.8) rectangle (9,4.8);
\draw[line width=1pt,color=black,smooth,samples=200,domain=-9:1] plot(\x,{1/((\x)-1.1)-0.2});
\draw[line width=1pt,color=black,smooth,samples=200,domain=1.11:9] plot(\x,{1/((\x)-1.1)-0.2});
\draw [line width=1.2pt,dash pattern=on 1pt off 2pt] (1.1,-4.8) -- (1.1,4.8);
\draw (2.66,1.78) node[anchor=north west] {$\beta(r)$};
\draw (1.28,0.0) node[anchor=north west] {$r_*$};
\draw (6.1,0.54) node[anchor=north west] {$r^*$};
\begin{scriptsize}
\draw [fill=black] (1.1,0) circle (2.5pt);
\draw [fill=black] (6.1,0) circle (2.5pt);
\end{scriptsize}
\end{axis}
\end{tikzpicture} 
  \caption{The behaviour of $\beta(r)$.}\label{beta}
\end{figure}

As $r^* - r_* =a_2(\theta-\theta^{-1}) = a_2(-\sqrt{\lambda^2-4}) >0$, we see that $r_*< r^*$. Also, differentiating with respect to $r$ we conclude that $\beta(r)$ is decreasing and take the value zero only for $r^*:= a_2\theta  > r_*$. Thus we conclude that $\beta(r)>0$ for $r \in (r_*, r^*)$. Figure \ref{beta} illustrates a typical behaviour of the function $\beta(r)$.\\

\noindent {\bf (c)} We recall that $r^*=a_2 \theta$, moreover, from Theorem \ref{thm:results_on_the_sequence} (d), we know that $\displaystyle \lim_{j \to \infty}a_{j}=\theta$ hence
\begin{equation}\label{eq:form_r_star}
  r^*=\lim_{j \to \infty} a_2 a_{j}.
\end{equation}
By Theorem \ref{th:T_j}, we know that $\rho(T_j) < \rho(T)$, where $T_j$ is the tree of Figure \ref{fig:T_j}. Now, we apply {\bf Diagonalize}($T_{j}$, $-\lambda$) with the root at $S_{r_2}$, for $\lambda =\rho(T)$. By our comparison method,  we see that
$$f_{T_{j}}(u)= -\lambda -\frac{1}{a_{j}}-\frac{r_2}{a_{2}}<0$$
or, equivalently, $a_{j+1} < \frac{r_2}{a_{2}},$ and since $a_{2}<0$, we get $r_2<a_{2}a_{j+1}$.\\
Taking the limit on both sides and using Equation~(\ref{eq:form_r_star}), we obtain
$$r_2 \leq \lim_{j \to \infty} a_2 a_{j+1}= \lim_{j \to \infty} a_2 a_{j}=r^*.$$
We also observe that $r_* < r_1$, because $r_* = \frac{a_2}{\theta}=2\,{\frac {{\lambda}^{2}-1}{\lambda\, \left( \lambda+\sqrt {{\lambda}^{2}-4} \right) }}<2\leq r_1$.

Since $z_{j}= \theta + \frac{\theta^{-1}-\theta}{\beta(r)(\theta^2)^j+1},$ for   $\theta^{-1}-\theta= \sqrt{\lambda^2-4}>0$ and $\theta <0$, we see that $z_{j}$ is always negative and decreasing, as long as $\beta(r) >0$.
Now, because $r_* < r_1 < r_2 \leq r^*$, we see from item (b), that $\beta(r) >0$. Moreover $z_j$ tends to $\theta$ as $j \to \infty$.
\end{proof}

\section{Ordering the 2-switches of $\mathfrak{F}(n)$}
\label{sec:ordering the family}

We show in this section how the spectral radius varies in each case for all the possible 2-switching positions in the appropriate interval. We notice that $\mathfrak{F}(n)$ is preserved by both types of $(s,t)$-2-switch, only changing $[h, q_1, q_2]$ (see Figure~\ref{Main_Family_Diago}).

\subsection{Warmup: Ordering 2-switches of Type I}

Given a 2-switch of Type I such that $T=[h, q_1, q_2] \to \tilde{T}=[h', q'_1, q'_2]$ (see Figure \ref{TypeISwitch}), we observe that the actual result of the operation in the tree is an increment (decrement) of the length  $h$ with a decrement (increment) of the length $q_1$, while $q_2$ remains unchanged.

In order to study the behavior of the spectral radius of members of this family, it is enough to study the 2-switch  $T=[h, q_1, q_2] \to \tilde{T}=[h'=h-1, q'_1=q_1 +1, q'_2=q_2]$, since this will cover all possible positions.

We will prove that $\rho(T) > \rho(\tilde{T})$ using our comparison method, hence we need to prove that if $-\lambda -\frac{1}{a_{h}}-\frac{1}{b_{q_1}}-\frac{1}{c_{q_2}} =0$ then
$-\lambda -\frac{1}{a_{h'}}-\frac{1}{b_{q'_1}}-\frac{1}{c_{q_2}}= -\lambda -\frac{1}{a_{h-1}}-\frac{1}{b_{q_1 +1}}-\frac{1}{c_{q_2}} <0$.
Notice that, from the first equation, we obtain $-\lambda -\frac{1}{c_{q_2}}= \frac{1}{a_{h}}+\frac{1}{b_{q_1}}$ and substituting in the second one it is equivalent to
$\frac{1}{a_{h}}+\frac{1}{b_{q_1}}-\frac{1}{a_{h-1}}-\frac{1}{b_{q_1 +1}} <0$, which in turn is equivalent to
\begin{equation}\label{eq: condit_typeII-decrease}
   (a_{h}- a_{h+1}) +  (b_{q_{1}+2}- b_{q_{1}+1}) <0.
\end{equation}

From Theorem~\ref{thm:results_on_the_sequence}, the sequence $a_j$ is increasing thus $a_{h}- a_{h+1}<0$. From Theorem~\ref{thm:results_on_the_sequence zjr}, the sequence $b_j$ is decreasing thus $b_{q_{1}+2}- b_{q_{1}+1}<0$.

We remark that the transformation $T=[h, q_1, q_2] \to \tilde{T}=[h'=h-1, q'_1=q_1, q'_2=q_2+1]$ is also of Type-I and using a similar argument, we can show that the index decreases as well.  This proves the following theorem.
\begin{theorem} \label{thm: ordering type I}
   Let $T=[h,q_1,q_2]$ be a tree in $\mathfrak{F}(n)$ and $\tilde{T}=[h',q'_1,q'_2]$ be the graph obtained by a 2-switch in Figure~\ref{TypeISwitch}. If $h'=h-1, q'_{1}=q_{1}+1$ and $q'_{2}=q_{2}$ or if  $h'=h-1, q'_{1}=q_{1}$ and $q'_{2}=q_{2}+1$ then $\rho(T)> \rho(\tilde{T})$.
\end{theorem}

We remark that this result may be obtained also by using Lemma \ref{lem:hoff} due to Hoffman \& Smith \cite{Hoff2}. We add an edge on the internal path from $u$ to $r_1P_2$ (or from $u$ to $r_2P_2$) and then erase the  pendant vertex from $P_h$, so that the spectral radius decreases even more, keeping both with the same number of vertices.

Our method, after we obtained that the sequence $b_j$ is decreasing, is simple enough to provide the alternative proof. We observe, however, that for 2-switching of Type II, we are not aware of a known result that apply. Additionally, or perhaps because of that, the application of our method  requires to overcome quite a few technical difficulties.

\subsection{Ordering 2-switches of Type II}

   We observe that a 2-switch of type II can be seen as a displacement of the central path of length $h$ from the position closest  to $r_1 \, P_2$ to the closest to $r_2 \, P_2$ (or vice-versa). Indeed, if we take $s=q_2$ then $q'_1 =q_2+t-s= t$ and $q'_2=q_1-t+q_2=q_1+q_2-t$, for $1 \leq t \leq q_1 -1$. For instance, taking $t=1$ we can apply the 2-switch sequentially.

Since every configuration $[h, q_1, q_2]$ can be obtained by successive changes by 1, we only need to consider the case where $T=[h, q_1, q_2] \to \tilde{T}=[h, q'_1=q_1 -1, q'_2=q_2+1]$. We will prove that this operation decreases the spectral radius

By using our method of Section~\ref{sec:family}, given a 2-switch of Type II we need to prove that if $-\lambda -\frac{1}{a_{h}}-\frac{1}{b_{q_1}}-\frac{1}{c_{q_2}} =0$ then $-\lambda -\frac{1}{a_{h}}-\frac{1}{b_{q_1 -1}}-\frac{1}{c_{q_2+1}} <0$.

\begin{theorem}\label{thm: ordering type II}
   Let $T=[h,q_1,q_2]$ be a tree obtained by a 2-switch in Figure~\ref{TypeIISwitch}. If $r_2 > r_1 \geq 2$, $q'_{1}=q_{1}-1$ and $q'_{2}=q_{2}+1$ then $\rho(T)> \rho(T')$.
\end{theorem}
\begin{proof}
We recall that we need to prove that, given a 2-switch of Type II, we need to prove that if $-\lambda -\frac{1}{a_{h}}-\frac{1}{b_{q_1}}-\frac{1}{c_{q_2}} =0$ then $-\lambda -\frac{1}{a_{h}}-\frac{1}{b_{q_1 -1}}-\frac{1}{c_{q_2+1}} <0$. We define
\begin{equation}\label{eq:ff1}
  \mathcal{I}:= -\frac{1}{a_{h}}+\left(-\lambda-\frac{1}{b_{q_1 -1}}\right)-\frac{1}{c_{q_2+1}}= -\frac{1}{a_{h}}+ b_{q_1}-\frac{1}{c_{q_2+1}}
\end{equation}
and, from the first equation, $ -\frac{1}{a_{h}}-\frac{1}{b_{q_1}}+\left(-\lambda-\frac{1}{c_{q_2}}\right) =0$ or $c_{q_2+1}=\frac{1}{a_{h}}+ \frac{1}{b_{q_1}}$. Substituting that in (\ref{eq:ff1}) we obtain
\begin{equation}\label{eq:ff2}
  \mathcal{I}= -\frac{1}{a_{h}}+ b_{q_1}-\frac{1}{\frac{1}{a_{h}}+ \frac{1}{b_{q_1}}}.
\end{equation}
Using the fact that $a_{h}<0$ and $b_{q_1}<0$ conclude that  $\mathcal{I}<0$ if and only if
\begin{equation}\label{eq:ff3}
  a_{h} ( b_{q_1}^2-1) <  b_{q_1}.
\end{equation}

We already know that $a_{h} < \theta <  b_{q_1}$ but we do not know the sign of $b_{q_1}^2 - 1 $. We claim that $b_{q_1}^2 - 1 >0$. To see that, we recall that $b_j$ is decreasing  and $b_1$ is given by  $b_1:=-\lambda - \frac{r_1}{-\lambda+\frac{1}{\lambda}}=-\lambda - \frac{2}{-\lambda+\frac{1}{\lambda}}= \phi(\lambda, r_1)$
where  the auxiliary function $\phi: A \to \mathbb{R}$ is given by
$$\phi(t,r):=-t - \frac{r}{-t+\frac{1}{t}},$$
defined on the set $A=\{(t,r) \, | \, t \geq 2, \, r \geq 2\} \subset \mathbb{R}^{2}$.
\begin{figure}[H]
  \centering
  \includegraphics[width=8cm, height=5.5cm]{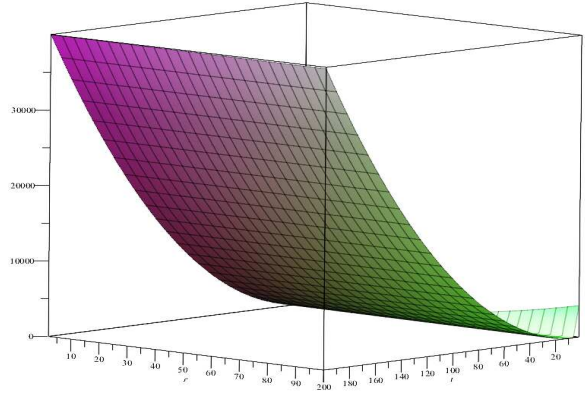}
  \caption{The graph of  $\phi(t,r)^2-1$ plotted in $A$. }\label{fig:phi}
\end{figure}
It is easy to see that $\phi(t,r)^2-1>0$ for $(t,r) \in A$, see Figure~\ref{fig:phi}. In other words, $b_{1}^2 - 1 >0$ and $b_{1}<-1$ because it is negative.
For any $T \in \mathfrak{F}(n)$ we conclude that $\theta < b_{q_1} < -1=b_{1}$, for any $T \in \mathfrak{F}(n)$. In particular $b_{q_1}^2 - 1 >0$.

From these facts, we can rewrite equation~(\ref{eq:ff3})
in the equivalent form
\begin{equation}\label{eq:ff4}
  a_{h} < \frac{ b_{q_1}}{ b_{q_1}^2-1}.
\end{equation}
In order to conclude our proof, it  is sufficient to prove that $\theta< \frac{ b_{q_1}}{ b_{q_1}^2-1}$ because $a_{h} < \theta$ for all $h$.
At this point, it is useful to introduce a second auxiliary function
$$\psi(t):=\frac{t}{t^2-1}, t<0.$$
This function is obviously decreasing in the interval $(-\infty, \, -1)$. As the correspondence $j \to b_j$ is also decreasing, we conclude that the correspondence $j \to \psi(b_j)$ is increasing and, as a consequence, $\frac{ b_{1}}{ b_{1}^2-1}< \frac{ b_{q_1}}{ b_{q_1}^2-1}$.

We claim that $\frac{ b_{1}}{ b_{1}^2-1} > \theta$ or equivalently
\begin{equation}\label{eq:ff4.1}
   \frac{ \left(-\lambda - \frac{r1}{-\lambda+\frac{1}{\lambda}}\right)}{ \left(-\lambda - \frac{r1}{-\lambda+\frac{1}{\lambda}}\right)^2-1}>  \frac{-\lambda -\sqrt{\lambda^2 -4}}{2},
\end{equation}
for $\lambda=\rho(T)$.

The inequality (\ref{eq:ff4.1}) is equivalent to $g(x, r)>0$ for $x \geq \rho(T) \geq \sqrt{r_2 + 2} \geq \sqrt{r_1 + 3}$ and $r=r_1 \geq 2$, where $g: B \to \mathbb{R}$ is given by
$$g(x,r)=\frac{ \left(-x - \frac{r}{-x+\frac{1}{x}}\right)}{ \left(-x - \frac{r}{-x+\frac{1}{x}}\right)^2-1} +  \frac{x +\sqrt{x^2 -4}}{2},$$
defined on $B= \{(x,r) \, | \, x \geq \sqrt{r + 3}, \, r \geq 2\} \subset \mathbb{R}^{2}$.
\begin{figure}[H]
  \centering
  \includegraphics[width=8cm, height=5.5cm]{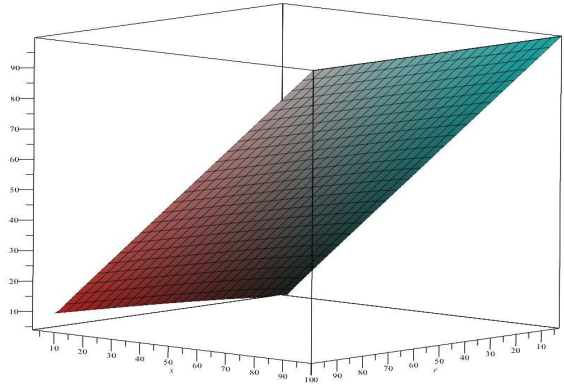}
  \caption{The  graph of $g$ on $B$. }\label{img_monster}
\end{figure}
As we can see in Figure~\ref{img_monster}, the function is always positive in this set concluding our proof.
\end{proof}

\section{Spectral radius ordering in $\mathfrak{F}(n)$} \label{sec: ordering}

In this section we provide a total ordering by the spectral radius in the family $\mathfrak{F}(n)$. In order to recall our notation, we notice that $r_1 >r_2\geq 2$ and the number $n$ of vertices  are fixed, hence each element in $\mathfrak{F}(n)$ is uniquely determined by the 3-uple $[h, q_1, q_2]$, that is,
$$\mathfrak{F}(n):=\{T=[h, q_1, q_2] \; | \; h + q_1 + q_2 =n-1-2 (r_1+r_2), \; h, q_1, q_2  \geq 2\}.$$
The degree sequence of an element $T \in \mathfrak{F}(n)$ is given by
$d:=[r_1 +1, r_2 +1, 3, 2^{n-(r_1+r_2) -2},1^{r_1+r_2+1}],$
for  a fixed pair $r_1 > r_2\geq 2$. We recall the well known result from \cite{FHM} (also \cite{Berge}):
\begin{theorem}
  If $G$ e $H$ have the same degree sequence then, there exists a 2-switch sequence transforming $G$ into $H$.
\end{theorem}
As a consequence, each two elements in $\mathfrak{F}(n)$ are transformed into another by a sequence of 2-switch transformations.

The remarkable fact is that, as we will see next, we can reach any element  from another using only two 2-switches. Additionally, the sequence of 2-switches is closed in the family, that is, any intermediate member is also a member. Given $T=[h, q_1, q_2] \in \mathfrak{F}(n)$ we define the following operations.
\begin{itemize}
  \item $\alpha: \mathfrak{F}(n) \to \mathfrak{F}(n)$, given by  $\alpha([h, q_1, q_2])= [h-1, q_1 +1, q_2]$, for $h \geq 3$;
  \item $\beta: \mathfrak{F}(n) \to \mathfrak{F}(n)$, given by  $\beta([h, q_1, q_2])= [h, q_1-1, q_2+1]$, for $q_1 \geq 3$.
  \item $\gamma: \mathfrak{F}(n) \to \mathfrak{F}(n)$, given by  $\beta([h, q_1, q_2])= [h-1, q_1, q_2+1]$, for $h \geq 3$.
\end{itemize}

We remark that the transformations make sense for $h, q_1,q_2 \geq 1$, and the results from Section \ref{sec:ordering the family} do apply. However, we observe that the original $T=[h, q_1,q_2]$ and the transformed tree $T'=[h', q'_1,q'_2]$ must have parameters $h,h', q_1,q'_1,q_2,q'_2\geq 2$, otherwise they will not be 2-switches, because the degree sequence changes.

\begin{theorem}\label{th:abc} Let $\alpha, \beta$ and $\gamma$ be the transformations defined in $\mathfrak{F}(n)$. Let $T^*=[h_{0},2, 2]$, where $h_{0}:=n-5-2(r_1+r_2) \geq 2$. The following facts are true
   \begin{itemize}
     \item[(a)] $\alpha, \beta$ and $\gamma$  are 2-switch transformations;
     \item[(b)] $\alpha, \beta$ and $\gamma$  are index decreasing transformations;
     \item[(c)] Any $T \in\mathfrak{F}(n)$ can be obtained from $T^*$ by a sequence of $\alpha$ and $\beta$ transformations;
     \item[(d)] Any $T \in\mathfrak{F}(n)$ can be obtained from $T^*$ by a sequence of $\alpha$ and $\gamma$ transformations.
   \end{itemize}
\end{theorem}

\begin{proof} From our definition of Section \ref{sec:family}, we see that $\alpha$ and $\gamma$ are Type-I 2-switches, while $\beta$ is a Type-II 2-switch and our results of Theorem~\ref{thm: ordering type I} and Theorem~\ref{thm: ordering type II} apply, hence they decrease the index. This proves (a) and (b).\\
(c) To see that, consider the tree $T^*=[h_{0} ,2, 2]$.  As $T^*=u + P_{h_{0}} \ast  S_{0} \oplus P_{2} \ast  S_{r_{1}} \oplus P_{2} \ast  S_{r_{2}},$ and  $h_{0}=n-5-2(r_1+r_2)  \geq 2$, because $n\geq 7+ 2 (r_1+r_2)$, we see that $T^* \in \mathfrak{F}(n)$. Let $T=[h',q'_1,q'_2]$ be any tree in $\mathfrak{F}(n)$ ($h', q'_1,q'_2 \geq 2$). Then we have the following\\
  $\alpha([h_{0} , 2, 2])= [h_{0}-1 , 2+1, 2],$   $\alpha([h_{0}-1 , 2+1, 2])= [h_{0}-2 , 2+2, 2],$
  and so on, until we obtain $h_{0} -k= h'$, that is
  $\alpha^{k}([h_{0} , 2, 2])= [h', 2+(h_{0}-h'), 2]$.\\
  Now we apply $\beta$ transformation $j$ times obtaining
  $\beta^{j}\alpha^{k}([h_{0} , 2, 2])= [h', 2+(h_{0} -h')-j, 2+j]$.\\
  At this point we claim that, making $2+(n-5-2(r_1+r_2) -h')-j= q'_1$ we get $2+j=q'_2$. Indeed, $2+(n-5-2(r_1+r_2) -h')-j= q'_1$ means that $j= 2+(n-5-2(r_1+r_2) -h' -q'_1) -q'_2 +q'_2= q'_2 -2 + ( n- (1+ h' +q'_1 +q'_2 +2(r_1+r_2))) =q'_2 -2$. Hence $2+j = q'_2$. Thus we conclude that there exist $k,j \in \mathbb{N}$ such that  $\beta^{j}\alpha^{k}(T^*)=T$ for any $T \in\mathfrak{F}(n)$.\\
(d) A similar reasoning as in (c) shows that $T^* \in \mathfrak{F}(n)$. We now apply $k'=q'_1-2$ times the transformation $\alpha$, arriving at $[h_0-k', q'_1,2]$. Reasoning as above we show that there is an integer $j'$ so that $h_0-k'-j'=h'$ and $q'_2=2+j'$. Hence applying now $j'$ times the transformation $\gamma$, shows that $\gamma^{j'}\alpha^{k'} (T^*) = T$.
\end{proof}

We observe that $T^*=[h_0,2,2]$ has the configuration with largest possible $h_0$. From Theorem \ref{th:abc} (b) and (c) (or from (b) and (d)), it follows that $T^*$ is the extremal element of $\mathfrak{F}(n)$: it has the maximum spectral radius.

The next result is quite remarkable in the sense that it provides a complete ordering of $\mathfrak{F}(n)$ using only $\alpha$ and $\gamma$ transformations, allowing us to find also the element of minimum index.
\begin{theorem}\label{thm:total_order}
   Consider the family $\mathfrak{F}(n)$, where $2 \leq r_1 < r_2 $,  $h_{0}=n-5-2(r_1+r_2) \geq 2$, $T^*=[h_{0},2,2] \in \mathfrak{F}(n)$ and $T_*=[2, 2, h_{0}] \in \mathfrak{F}(n)$. The following claims are true.
   \begin{itemize}
     \item[(a)]  The ordered sequence  $$\mathcal{A}:=\{T^*,\alpha(T^*), \cdots , \alpha^{h_{0}-2}(T^*),\gamma(T^*),\alpha(\gamma(T^*)), \cdots , \alpha^{h_{0}-3}(\gamma(T^*)), \cdots, \gamma^{h_{0}-2}(T^*)=T_*\}$$ is equal to $\mathfrak{F}(n)$;
     \item[(b)] The sequence $\mathcal{A}$ is ordered by the inverse lexicographic order:\\
         $[x,y,z] \succ  [x',y',z']$ iff $z'>z$ or $z'=z$ but $y'>y$.
     \item[(c)] If $[x,y,z] \succ  [x',y',z']$ then $\rho([x,y,z]) > \rho([x',y',z'])$;
     \item[(d)] The maximum (resp. minimum) index in $\mathfrak{F}(n)$ is $\rho(T^*)$ (resp. $\rho(T_*)$).
   \end{itemize}
\end{theorem}

Before we prove Theorem~\ref{thm:total_order} we need a technical lemma.

\begin{lemma}\label{lem:monot_gamma} If $T^*=[h_{0}, 2 , 2] \in \mathfrak{F}(n)$, then
$\rho(\alpha^{h_{0}-j - 2}(\gamma^{j}(T^*)) > \rho(\gamma^{j+1}(T^*))$, for $j=0,1,2,..., h_{0}- 2$, that is,
$$\rho([2, h_{0} -j, 2]) > \rho([h_{0} -(j+1), 2 , 2+(j+1)]).$$
\end{lemma}
\begin{proof}
We will consider the first case $j=0$. The rest of the cases are identical. Thus, we must prove that
$$\rho([2, h_{0}, 2]) > \rho([h_{0} -1, 2 , 3]).$$

Taking $\lambda=\rho([2, h_{0}, 2])$ and applying \texttt{Diagonalize}($[2, h_{0}, 2], -\lambda$), we have
$$-\lambda -\frac{1}{a_{2}} -\frac{1}{b_{h_0}} -\frac{1}{c_{2}} =0.$$
From this, we obtain $c_{3}= \frac{1}{\frac{1}{a_{2}} +\frac{1}{b_{h_0}}}$.

Analogously, applying \texttt{Diagonalize}($[h_{0} -1, 2 , 3], -\lambda$), we have
$$-\lambda -\frac{1}{a_{h_0 -1}} -\frac{1}{b_{2}} -\frac{1}{c_{3}} =:\mathcal{I}.$$
We need to show that $\mathcal{I}<0$. Writing $\mathcal{I}= a_{h_0 }-\frac{1}{b_{2}} -\frac{1}{\frac{1}{a_{2}} +\frac{1}{b_{h_0}}}$, we conclude that $\mathcal{I}<0$ if the function
$$g(\lambda, r_1, h_{0}):= a_{h_0 }-\frac{1}{b_{2}} -\frac{1}{\frac{1}{a_{2}} +\frac{1}{b_{h_0}}}$$
is always negative for $r_1 \geq 2$, $h_{0} \geq 2$ and $\lambda \geq \sqrt{r_1 + 3}$.

At a first glance, we can not plot a graph as we did before because we have three variables. However, if we consider the variation of the variable $h_{0}$ we observe that the correspondence $h_{0} \to g(\lambda, r_1, h_{0})$ is monotonously increasing because
$$\frac{d}{dh_{0}}g(\lambda, r_1, h_{0}) =\frac{d a_{h_0 }}{dh_{0}} - \frac{1}{\left(\frac{1}{a_{2}} +\frac{1}{b_{h_0}}\right)^2} \frac{1}{b_{h_0}^2}\frac{d b_{h_0 }}{dh_{0}}>0,$$
since $\frac{d a_{h_0 }}{dh_{0}}>0$ and $\frac{d b_{h_0 }}{dh_{0}}<0$.

Hence we just need to show that the limit function
$$f(\lambda, r_1):=\lim_{h_{0}\to \infty }g(\lambda, r_1, h_{0})= \theta - \frac{1}{b_{2}} -\frac{1}{\frac{1}{a_{2}} +\frac{1}{\theta}}$$
is always negative for $C:=\{(\lambda,r1)\,|\, r_1 \geq 2, \, \lambda \geq \sqrt{r_1 + 3}\}$, as shown on Figure~\ref{fig:total_order}, concluding our proof.
\begin{figure}[H]
  \centering
  \includegraphics[width=6cm, height=4.5cm]{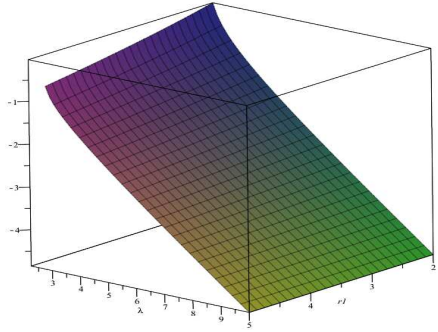}
  \caption{The  graph of $f$ on $C$. }\label{fig:total_order}
\end{figure}
\end{proof}

\begin{proof}\textbf{(of Theorem~\ref{thm:total_order})}\\
To see (a) we first see that both transformations $\alpha$ and $\gamma$ decrease the index. By Theorem~\ref{th:abc} (d), we can reach any configuration from $T^*$, showing that every configuration appears in the sequence $\mathcal{A}$.\\
For (b), we notice that the inverse lexicographic order:\\
$[x,y,z] \succ  [x',y',z']$ if an only if $z'>z$ or $z'=z$ but $y'>y$,
is naturally produced in $\mathcal{A}$. For instance comparing $[x,y,z]$ with $[x',y',z']= \alpha([x,y,z])= [x-1, y+1, z]$  we obtain $z'=z$ but $y'>y$. The only possible difficulty is to compare, for example, $[x,y,z]= \alpha^{h_{0}-2}(T^*)= [2, h_{0}-2,2] $ with $[x',y',z']=\gamma(T^*)= [h_{0}-1, 2, 3]$, in this case we have $z'>z$.\\
For (c) we can use the same reasoning, if the next element is obtained from the previous one by $\alpha$ the index decreases by Theorem~\ref{thm: ordering type I}. Again, it remains to analyse the case  $[x,y,z]= \alpha^{h_{0}-2}(T^*)= [2, h_{0}-2,2] $ and $[x',y',z']=\gamma(T^*)= [h_{0}-1, 2, 3]$. In this case the index decrease by Lemma~\ref{lem:monot_gamma}.\\
The item (d) is a direct consequence of the previous items.
\end{proof}

\begin{example} \label{ex: final_order_total} We consider $\mathfrak{F}(23)$, with $r_1=2$, $r_2=3$ and $h_{0}=n-5-2(r_1+r_2)=8$. In this case $T^*=[h_{0},2,2]=[8,2,2]$ and $T_*=[2,2, h_{0}]=[2,2,8]$. We will use the same procedure described in the proof of Theorem~\ref{thm:total_order} to build a table where we show the spectral radius of each intermediary tree:\\

    \begin{center}
    \begin{tabular}{|c|c|c|}
      \hline
      Transf. / $\mathcal{A}$ & Tree & Index\\
      \hline
       $T^*$    &[8, 2, 2]& 2.31431268823172996316982502630\\
       $\alpha(T^*)$    &[7, 3, 2]& 2.30752321205156788164155922354\\
       $\alpha^{2}(T^*)$      &[6, 4, 2]& 2.30509257122848263666229555974\\
       $\alpha^{3}(T^*)$     &[5, 5, 2]& 2.30414417603593895847264293478\\
       $\alpha^{4}(T^*)$     &[4, 6, 2]& 2.30348720654135784657134525755\\
       $\alpha^{5}(T^*)$     &[3, 7, 2]& 2.30226165044440472718571097461\\
       $\alpha^{6}(T^*)$     &[2, 8, 2]& 2.29881642949995094980856594643\\
       $\gamma(T^*)$     &[7, 2, 3]& 2.28520768467980257500859073365\\
       $\alpha(\gamma(T^*))$    &[6, 3, 3]& 2.28076523286917478390041282633\\
       $\alpha^{2}(\gamma(T^*))$    &[5, 4, 3]& 2.27913084342903308996825366690\\
       $\alpha^{3}(\gamma(T^*))$    &[4, 5, 3]& 2.27834791245706879712729095155\\
       $\alpha^{4}(\gamma(T^*))$    &[3, 6, 3]& 2.27748824925244285093685838480\\
       $\alpha^{5}(\gamma(T^*))$    &[2, 7, 3]& 2.27554403106324050144208754160\\
       $\gamma^{2}(T^*)$    &[6, 2, 4]& 2.27010998510725135104117051475\\
       $\alpha(\gamma^{2}(T^*))$    &[5, 3, 4]& 2.26762484634172519636930335282\\
       $\alpha^{2}(\gamma^{2}(T^*))$    &[4, 4, 4]& 2.26667762008239070931668388638\\
       $\alpha^{3}(\gamma^{2}(T^*))$    &[3, 5, 4]& 2.26605728367174815669677409819\\
       $\alpha^{4}(\gamma^{2}(T^*))$    &[2, 6, 4]& 2.26506821261118740374393886088\\
       $\gamma^{3}(T^*)$    &[5, 2, 5]& 2.26290253458453744084697620016\\
       $\alpha(\gamma^{3}(T^*))$    &[4, 3, 5]& 2.26171078345443097808224085587\\
       $\alpha^{2}(\gamma^{3}(T^*))$    &[3, 4, 5]& 2.26119844487818869745804831320\\
       $\alpha^{3}(\gamma^{3}(T^*))$    &[2, 5, 5]& 2.26069897200749878293447592468\\
       $\gamma^{4}(T^*)$    &[4, 2, 6]& 2.25980268994372236598891968054\\
       $\alpha(\gamma^{4}(T^*))$    &[3, 3, 6]& 2.25927957177517211016191460326\\
       $\alpha^{2}(\gamma^{4}(T^*))$    &[2, 4, 6]& 2.25898741243972985580277387992\\
       $\gamma^{5}(T^*)$    &[3, 2, 7]& 2.25857320563154910353684549335\\
       $\alpha(\gamma^{5}(T^*))$    &[2, 3, 7]& 2.25834278165321357168906331906\\
       $\gamma^{6}(T^*)=T_*$    &[2, 2, 8]& 2.25810972712429442797185863240\\
      \hline
    \end{tabular}
    \end{center}

    As expected, the index ordering in $\mathfrak{F}(23)$ is total.
\end{example}

\begin{remark}\label{rem:applic_total_order}
Theorem~\ref{thm:total_order} has a powerful application, not only we can obtain the extremal indices in $\mathfrak{F}(n)$, but given any $[x,y,z] \in \mathfrak{F}(n)$ and a 2-switch $F$, such that $F([x,y,z])=[x',y',z']$ we can immediately say if $F$ increases or decreases the index by using the order relation in Theorem~\ref{thm:total_order} (c).
\end{remark}

\section{Additional considerations on the 2-switch spectral ordering problem} \label{sec: additional considerations}

We believe that transformations by 2-switches which are monotonous with respect to the spectral radius  consist of a powerful tool for spectral graph theory problems. We just recall that, following \cite{Lin2006} we have the following operation which was a central piece to prove that all the starlike trees are completely ordered by their index in our recent work \cite{Oli18}.
\begin{lemma}\label{alpha}
\cite{Cve97,Lin2006}. Let $u$ be a vertex of a non-trivial connected graph $G$, and let $G^0_{k,l}$ denote the graph obtained from $G$ by adding pendant paths of length $k$ and $l$ at $u$. If $k \geq l \geq 1$, then $\lambda_1(G^0_{k,l})> \lambda_1(G^0_{k+1,l-1})$.
\end{lemma}

We can rewrite this result as a 2-switch spectral ordering problem. Define $\mathfrak{G}$ the family of all trees $T$ given by Figure~\ref{fig:new_fam},
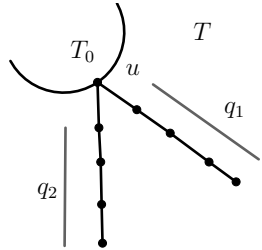
\begin{figure}[H]
  \centering
\definecolor{wrwrwr}{rgb}{0.3803921568627451,0.3803921568627451,0.3803921568627451}
\begin{tikzpicture}[line cap=round,line join=round,>=triangle 45,x=1cm,y=1cm,scale=0.6, every node/.style={scale=0.8}]
\clip(-9,-3.0) rectangle (9,3.3);
\draw [shift={(-2.2,2.66)},line width=1pt]  plot[domain=-2.6446240487107397:0.49696860487905337,variable=\t]({1*1.3423859355639864*cos(\t r)+0*1.3423859355639864*sin(\t r)},{0*1.3423859355639864*cos(\t r)+1*1.3423859355639864*sin(\t r)});
\draw [line width=1pt] (-1.4507114586748788,1.5461927088410352)-- (-0.58,0.94);
\draw [line width=1pt] (-0.58,0.94)-- (0.16,0.42);
\draw [line width=1pt] (0.16,0.42)-- (1.02,-0.22);
\draw [line width=1pt] (-1.4507114586748788,1.5461927088410352)-- (-1.42,0.54);
\draw [line width=1pt] (-1.42,0.54)-- (-1.38,-0.22);
\draw [line width=1pt] (-1.38,-0.22)-- (-1.36,-1.16);
\draw [line width=1pt] (-1.36,-1.16)-- (-1.34,-2.02);
\draw [line width=1pt] (1.02,-0.22)-- (1.62,-0.66);
\draw (-1.02,2.14) node[anchor=north west] {$u$};
\draw (-2.24,2.66) node[anchor=north west] {$T_0$};
\draw (0.48,3.06) node[anchor=north west] {$T$};
\draw [line width=1pt,color=wrwrwr] (-2.16,0.54)-- (-2.18,-2.08);
\draw [line width=1pt,color=wrwrwr] (-0.22,1.5)-- (2.12,-0.16);
\draw (1.16,1.3) node[anchor=north west] {$q_1$};
\draw (-2.98,-0.5) node[anchor=north west] {$q_2$};
\begin{scriptsize}
\draw [fill=black] (-1.4507114586748788,1.5461927088410352) circle (2.5pt);
\draw [fill=black] (-0.58,0.94) circle (2.5pt);
\draw [fill=black] (0.16,0.42) circle (2.5pt);
\draw [fill=black] (1.02,-0.22) circle (2.5pt);
\draw [fill=black] (-1.42,0.54) circle (2.5pt);
\draw [fill=black] (-1.38,-0.22) circle (2.5pt);
\draw [fill=black] (-1.36,-1.16) circle (2.5pt);
\draw [fill=black] (-1.34,-2.02) circle (2.5pt);
\draw [fill=black] (1.62,-0.66) circle (2.5pt);
\end{scriptsize}
\end{tikzpicture} 
  \caption{The family of trees $\mathfrak{G}$.}\label{fig:new_fam}
\end{figure}
with the following constraints.
\begin{itemize}
  \item[(1)] $q_1, q_2 \geq 2$;% and $q_1 \neq q_2$;
  \item[(2)] $T_0$ is any fixed tree.
\end{itemize}
Notice that, in this case, $n=|T|= |T_0| + q_1 + q_2$.

If we are interested in to the maximization/minimization of the index (spectral radius) of trees with a fixed degree sequence, in $\mathfrak{F}(n)$ we get the degree sequence $$d:=[d_{T_0},  2^{(q_1+q_2) -2},1^{2}],$$
where $d_{T_0}$ is the sequence of degrees in $T_0$ including the root $u$.
In this way each element in $\mathfrak{G}$, with $n$ vertices and degree sequence $d$ is uniquely determined by the pair $[q_1, q_2]$ that is,
$$\mathfrak{G}(n):=\{[q_1, q_2] \; | \; n=|T_0| + q_1 + q_2\}.$$
To preserve the degree sequence we can do a 2-switch as in Figure~\ref{fig:new_fam_2_switch}, obtaining a new tree $T'$.

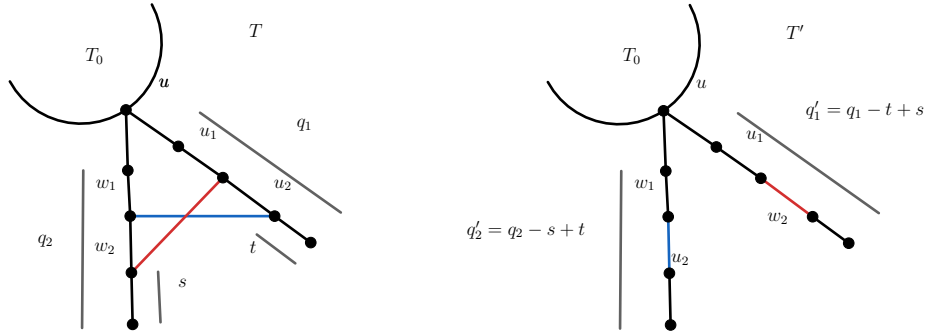
\begin{figure}[H]
  \centering
\definecolor{wrwrwr}{rgb}{0.3803921568627451,0.3803921568627451,0.3803921568627451}
\definecolor{rvwvcq}{rgb}{0.08235294117647059,0.396078431372549,0.7529411764705882}
\definecolor{dtsfsf}{rgb}{0.8274509803921568,0.1843137254901961,0.1843137254901961}
\begin{tikzpicture}[line cap=round,line join=round,>=triangle 45,x=1cm,y=1cm,scale=0.8, every node/.style={scale=0.6}]
\clip(-5.0,-3.0) rectangle (13.0,3.5);
\draw [shift={(6.74,2.65)},line width=1pt]  plot[domain=-2.6446240487107397:0.49696860487905337,variable=\t]({1*1.3423859355639873*cos(\t r)+0*1.3423859355639873*sin(\t r)},{0*1.3423859355639873*cos(\t r)+1*1.3423859355639873*sin(\t r)});
\draw [line width=1pt] (7.48209399038635,1.5313863448748422)-- (8.36,0.93);
\draw [line width=1pt] (8.36,0.93)-- (9.1,0.41);
\draw [line width=1pt,color=dtsfsf] (9.1,0.41)-- (9.96,-0.23);
\draw [line width=1pt] (7.48209399038635,1.5313863448748422)-- (7.52,0.53);
\draw [line width=1pt] (7.52,0.53)-- (7.56,-0.23);
\draw [line width=1pt,color=rvwvcq] (7.56,-0.23)-- (7.58,-1.17);
\draw [line width=1pt] (7.58,-1.17)-- (7.6,-2.03);
\draw [line width=1pt] (9.96,-0.23)-- (10.56,-0.67);
\draw (7.916720867208668,2.1479691388750606) node[anchor=north west] {$u$};
\draw (6.696080415906196,2.6833377578673727) node[anchor=north west] {$T_0$};
\draw (9.415753000387143,3.068803163541837) node[anchor=north west] {$T'$};
\draw [line width=1pt,color=wrwrwr] (6.78,0.53)-- (6.76,-2.09);
\draw [line width=1pt,color=wrwrwr] (8.72,1.49)-- (11.06,-0.17);
\draw (9.73697417178253,1.805333222719981) node[anchor=north west] {$q'_1=q_1 - t+s$};
\draw (4.104896299983404,-0.18623803993141913) node[anchor=north west] {$q'_2=q_2 -s + t$};
\draw [shift={(-2.2,2.66)},line width=1pt]  plot[domain=-2.6446240487107397:0.49696860487905337,variable=\t]({1*1.3423859355639864*cos(\t r)+0*1.3423859355639864*sin(\t r)},{0*1.3423859355639864*cos(\t r)+1*1.3423859355639864*sin(\t r)});
\draw [line width=1pt] (-1.4507114586748788,1.5461927088410352)-- (-0.58,0.94);
\draw [line width=1pt] (-0.58,0.94)-- (0.16,0.42);
\draw [line width=1pt] (0.16,0.42)-- (1.02,-0.22);
\draw [line width=1pt] (-1.4507114586748788,1.5461927088410352)-- (-1.42,0.54);
\draw [line width=1pt] (-1.42,0.54)-- (-1.38,-0.22);
\draw [line width=1pt] (-1.38,-0.22)-- (-1.36,-1.16);
\draw [line width=1pt] (-1.36,-1.16)-- (-1.34,-2.02);
\draw [line width=1pt] (1.02,-0.22)-- (1.62,-0.66);
\draw (-1.0132276975831014,2.169383883634753) node[anchor=north west] {$u$};
\draw (-2.2338681488855734,2.6833377578673727) node[anchor=north west] {$T_0$};
\draw (0.485804435595373,3.0902179083015295) node[anchor=north west] {$T$};
\draw [line width=1pt,color=wrwrwr] (-2.16,0.54)-- (-2.18,-2.08);
\draw [line width=1pt,color=wrwrwr] (-0.22,1.5)-- (2.12,-0.16);
\draw (1.2781499917039951,1.5055267960842864) node[anchor=north west] {$q_1$};
\draw (-3.0262137049941957,-0.42180023228803637) node[anchor=north west] {$q_2$};
\draw [line width=1pt,color=rvwvcq] (-1.38,-0.22)-- (1.02,-0.22);
\draw [line width=1pt,color=dtsfsf] (0.16,0.42)-- (-1.36,-1.16);
\draw (-2.0625501908080337,0.49903379237874007) node[anchor=north west] {$w_1$};
\draw (-1.0132276975831014,2.169383883634753) node[anchor=north west] {$u$};
\draw (-0.33522316243570928,1.3556235827664391) node[anchor=north west] {$u_1$};
\draw (8.744628615673908,1.312794093247054) node[anchor=north west] {$u_1$};
\draw (-2.083964935567726,-0.5502887008461912) node[anchor=north west] {$w_2$};
\draw (0.9,0.49903379237874007) node[anchor=north west] {$u_2$};
\draw (6.888813118743429,0.49903379237874007) node[anchor=north west] {$w_1$};
\draw (7.5,-0.7430214036834235) node[anchor=north west] {$u_2$};
\draw (9.1,-0.05774957137326427) node[anchor=north west] {$w_2$};
\draw [line width=1pt,color=wrwrwr] (-0.916666666666671,-1.145555555555553)-- (-0.8766666666666709,-1.985555555555552);
\draw [line width=1pt,color=wrwrwr] (0.7233333333333283,-0.5255555555555537)-- (1.363333333333328,-1.005555555555553);
\draw (-0.6920065261877141,-1.1499015541175805) node[anchor=north west] {$s$};
\draw (0.5,-0.5288739560864987) node[anchor=north west] {$t$};
\begin{scriptsize}
\draw [fill=black] (7.48209399038635,1.5313863448748422) circle (2.5pt);
\draw [fill=black] (8.36,0.93) circle (2.5pt);
\draw [fill=black] (9.1,0.41) circle (2.5pt);
\draw [fill=black] (9.96,-0.23) circle (2.5pt);
\draw [fill=black] (7.52,0.53) circle (2.5pt);
\draw [fill=black] (7.56,-0.23) circle (2.5pt);
\draw [fill=black] (7.58,-1.17) circle (2.5pt);
\draw [fill=black] (7.6,-2.03) circle (2.5pt);
\draw [fill=black] (10.56,-0.67) circle (2.5pt);
\draw [fill=black] (-1.4507114586748788,1.5461927088410352) circle (2.5pt);
\draw [fill=black] (-0.58,0.94) circle (2.5pt);
\draw [fill=black] (0.16,0.42) circle (2.5pt);
\draw [fill=black] (1.02,-0.22) circle (2.5pt);
\draw [fill=black] (-1.42,0.54) circle (2.5pt);
\draw [fill=black] (-1.38,-0.22) circle (2.5pt);
\draw [fill=black] (-1.36,-1.16) circle (2.5pt);
\draw [fill=black] (-1.34,-2.02) circle (2.5pt);
\draw [fill=black] (1.62,-0.66) circle (2.5pt);
\end{scriptsize}
\end{tikzpicture} 
  \caption{2-Switch from $T$ to $T'$.}\label{fig:new_fam_2_switch}
\end{figure}
Taking $\lambda$ as the index of $T$ and applying the Diagonalize algorithm to the trees $T$ and $T'$ we get $-\lambda -\xi-\frac{1}{a_{q_1}}-\frac{1}{a_{q_2}} =0$, where $\xi$ is the contribution of $T_0$ and,  $-\lambda -\xi-\frac{1}{a_{q'_1}}-\frac{1}{a_{q'_2}} <0$ if an only if $\rho(T)> \rho(T')$.

Assuming, $q'_1=q_1 -1$, $q'_2=q_2 +1$ and $q_1<q_2 +1$ we get, substituting  $-\lambda -\xi= \frac{1}{a_{q_1}}+\frac{1}{a_{q_2}}$ in the inequality, the equivalency
$$ \frac{1}{a_{q_1}}+\frac{1}{a_{q_2}}-\frac{1}{a_{q_1 -1}}-\frac{1}{a_{q_2 +1}}<0.$$
Adding and subtracting $-\lambda$ and using the formula~\eqref{b_j_c_j_def} we obtain that $\rho(T)> \rho(T')$ if and only if
\begin{equation}\label{eq: condit_new_type}
   a_{q_1} + \frac{1}{a_{q_1}}  < a_{q_2+1} + \frac{1}{a_{q_2+1}}.
\end{equation}
We notice that $a_j$ is increasing and $a_j<\theta <-1$. The function $\phi(x)=x+\frac{1}{x}$ is increasing in the interval $(-\infty, \, -1)$. Also, $a_{q_1} < a_{q_2+1}$ because we assume $q_1<q_2 +1$. Thus $\phi(a_{q_1}) < \phi(a_{q_2+1})$, which is exactly the claim in the equation~(\ref{eq: condit_new_type}).  We have proved the following result which is equivalent to Lemma~\ref{alpha}.
\begin{theorem} \label{thm: ordering new type}
   Let $T$ be the graph in Figure~\ref{fig:new_fam} and $T'$ be the graph obtained by a 2-switch in Figure~\ref{fig:new_fam_2_switch}. If $q_1<q_2 +1$, $q'_{1}=q_{1}-1$ and $q'_{2}=q_{2}+1$ then $\rho(T)> \rho(T')$.
\end{theorem}

\section*{Acknowledgments}
This work had the financial support of project MATHAmSud 88881.143281/2017-01. Victor N. Schvll\"{o}ner was partially supported by PROICO 03-0918 UNSL. V. Trevisan acknowledges partial support of CNPq grants  409746/2016-9 and 303334/2016-9, CAPES-Print 88887.467572/2019-00  and FAPERGS project PqG 17/2551-0001.

\end{document}